\documentclass[11pt]{amsart}
\usepackage[left=1.4in, right=1.4in]{geometry}
\usepackage{amssymb}
\usepackage{amscd}
\usepackage{graphicx,psfrag,epsfig,multirow}
\usepackage{amssymb,amsmath,amscd,amsthm,verbatim,color,mathrsfs,setspace,hyperref}
\numberwithin{equation}{section}
\usepackage[active]{srcltx}
\usepackage{geometry}
\newtheorem{Thm}{Theorem}[section]
\newtheorem{Def}{Definition}[section]
\newtheorem{Lm}[Thm]{Lemma}

\newtheorem{Prop}[Thm]{Proposition}
\newtheorem{Rk}{Remark}[section]
\newtheorem{Cor}[Thm]{Corollary}

\def\Diff{\mathrm{Diff}}
\def\Z{\mathbb Z}
\def\T{\mathbb T}
\def\C{\mathbb C}
\def\N{\mathbb N}
\def\S{\mathbb S}

\def\R{\mathbb R}

\def\eps{\varepsilon}
\def\al{\alpha}

\def\brho{\boldsymbol{\rho}}
\def\dt{\delta}

\def\G{\Gamma}
\def\lb{\lambda}

\def\id{\mathrm{id}}

\def\cF{\mathcal F}
\def\cA{\mathcal A}
\def\cF{\mathcal F}

\def\cH{\mathcal H}

\def\cW{\mathcal W}

\title[Rigidity of ABC solvable group actions]{Rigidity of some abelian-by-cyclic solvable group actions on $\T^N$
}
\author{Amie Wilkinson, Jinxin Xue}
\date\today
\address{Department of mathematics, the University of Chicago, Chicago, IL, US, 60637}
\email{wilkinso@math.uchicago.edu}%

\address{Department of mathematics \& Yau Mathematical Sciences Center, Tsinghua University, Beijing, China, 100084}
\email{jxue@mail.tsinghua.edu.cn}%

\begin{document}

\maketitle 
\begin{abstract}
In this paper, we study a natural class of  groups that act as affine transformations of  $\T^N$.  We investigate whether these solvable, ``abelian-by-cyclic,"  groups can act smoothly {\em and nonaffinely} on $\T^N$ while remaining homotopic to the affine actions.   In the affine actions, elliptic and hyperbolic dynamics coexist, forcing {\em a priori} complicated dynamics in nonaffine perturbations.   We first show, using the KAM method, that any small and sufficiently smooth perturbation of such an affine action can be conjugated smoothly to an affine action,  provided certain Diophantine conditions on the action are met. In dimension two,  under natural dynamical hypotheses, we get a complete classification of such actions; namely, any such group action by $C^r$ diffeomorphims can be conjugated to the affine action by $C^{r-\eps}$ conjugacy. Next, we show that in any dimension, $C^1$ small perturbations can be conjugated to an affine action via $C^{1+\eps}$ conjugacy. The method is a generalization of the Herman theory for circle diffeomorphisms to higher dimensions in the presence of a foliation structure provided by the hyperbolic dynamics. 
\end{abstract}

\begin{spacing}{0.3}
\tableofcontents
\end{spacing}
\renewcommand\contentsname{Index}

\section{Introduction}\label{SIntro}

This paper is motivated by an attempt to understand the action on the $2$-torus $\T^2 := \R^2/\Z^2$ generated by the diffeomorphisms $$g_0(x,y)=(2x+y,x+y),\quad g_1(x,y)=(x+\rho,y),\quad g_2(x,y)=(x,y+\rho),\quad \rho\in \R.$$
The map $g_0$ is a hyperbolic linear automorphism, and $g_1,g_2$ are translations. They satisfy the group relations $$g_0g_1=g_1^2g_2g_0,\quad g_0g_2=g_1g_2g_0,\quad g_1g_2=g_2g_1,$$ and no other relations if $\rho$ is irrational. Broadly stated, our aim is classify all diffeomorphisms $g_0,g_1,g_2$ satisfying these relations and no other. 

To place the problem in a more general context, in this paper, we establish  rigidity properties of certain solvable group actions on the torus $\T^N =\R^N/\Z^N$, for $N > 1$. 
The solvable groups $\Gamma$ considered here are the finitely presented, torsion-free, {\em abelian-by-cyclic} (ABC) groups, which admit a short exact sequence
\[0 \hookrightarrow \Z^d  \rightarrow \Gamma \rightarrow \Z \rightarrow 0.
\]
All such groups are of the form $\G = \G_{B}$, where  $B = (b_{ij}) $  is an integer valued, $d\times d$ matrix with $\det(  B)\neq 0$, and 
\begin{equation}
\G_{  B}=  \Z \ltimes \Z^d = \left\langle g_0,g_1,\ldots,g_d\ |\ g_0g_i=\left(\prod_{j=1}^d g_j^{b_{ji}}\right)g_0,\quad [g_i,g_j]=1,\quad i,j=1,2,\ldots,d\right\rangle.
\end{equation}
ABC groups have been studied intensively in geometric group theory, as they present the first case in the  open problem to classify finitely generated solvable groups up to quasi-isometry. The classification problem for ABC  groups has been solved in \cite{FM1} (the non-polycyclic case, $|\det   B| > 1$) and \cite{EFW1,EFW2} (the polycyclic case, $|\det   B| = 1$), where the authors also revealed  close connections between the geometry of these groups and dynamics \cite{FM2,EF}.  Here we consider actions of polycyclic ABC groups.  

A $C^r$ action $\al$ of a finitely generated group $\G$ with
generators $g_1,\ldots,g_k$ on a closed manifold $M$ is a homomorphism $\al\colon\G\to \mathrm{Diff}^r(M)$, where $\mathrm{Diff}^r(M)$ denotes the group of orientation-preserving, $C^k$ diffeomorphisms of $M$. The action is determined completely by $\al(g_1),\ldots,\al(g_k)$. 
The polycylic ABC groups admit natural affine actions on tori, as follows.

Up to rearranging the standard basis for $\R^d$, every matrix $  B\in \mathrm{ SL}(d; \Z)$ can be written in the form
\[   B = \begin{pmatrix}\bar{  B} & 0\\
0 & I_{d-N}
\end{pmatrix},
\]
for some $N\leq d$,  where $\bar {  B}=(\bar b_{ij})\in  \mathrm{ SL}(N; \Z)$, and  $I_{d-N}$ is the $(d-N)\times (d-N)$ identity matrix, chosen to be maximal.  
To avoid obviously degenerate actions, we restrict our attention to the cases where $d = KN+1$, for some $K\geq 1$.  Then 
\begin{equation}\label{EqGrpRelation}
\begin{aligned}
\G_{  B} = \G_{\bar {  B},K}&:=\langle g_0,g_{i,k},\ i=1,\ldots,N,\ k=1,\ldots,K\ |\ \ [g_{i,k},g_{j,\ell}]=1,\\
& \left.g_0g_{i,k}=\left(\prod_{j=1}^N g_{j,k}^{\bar b_{ji}}\right)g_0,\quad i,j=1,\ldots,N,\ k,\ell=1,\ldots,K\right\rangle.
\end{aligned}
\end{equation}
Note that $\G_{\bar {  B}} = \G_{\bar {  B}, 1}$.

In the affine actions of $\G_{\bar {  B},K}$ we consider, the element $g_0$ acts on $\T^N$ by the
automorphism $x\mapsto \bar Ax\,\, (\mathrm{mod }\, \,\Z^N)$  induced by $\bar A\in \mathrm{SL}(N,\Z)$, and the elements $g_{i,k},  i=1,\ldots,N,\ k=1,\ldots,K$ act as translations
$x\mapsto x + \rho_{i,j}\,\,  (\mathrm{mod } \,\Z^N)$, where $\rho_{i,j}\in \R^N$.   Thus if we denote this action by $\al=\al_{\bar B, K} \colon\G_{\bar B,K} \to \mathrm{Diff}^r(\T^N) $, we have
\[\al (g_0)(x) = \bar A x,\quad\hbox{and } \al (g_{i,k})(x) =  x + \rho_{i,j}.
\]
The group relations in $\G_{\bar B, K}$ restrict the possible values of $\rho_{i,j}$; we describe precisely these restrictions in the next subsection.  We will see that for a typical $\bar A$, the affine actions define a finite dimensional space of distinct (i.e. nonconjugate) actions on the torus.  

Given such a group $\G_{\bar B, K}$ with the associated affine action $\bar\alpha $ we investigate whether there exist other actions $\al \colon\G_{\bar B,K} \to \Diff^r(\T^N)$ that are homotopic to  $\bar\alpha $ but not conjugate to  $\bar\alpha $ in the group $\mathrm{Diff}^r(\T^N)$.  
If there are no such actions, or if such actions are proscribed in some manner, then the group is colloquially said to be {\em rigid} (a much more precise definition is given below).

The main rigidity results of this paper can be grouped into two classes: local and global.  Loosely speaking, local rigidity results concern those $C^r$ actions that are $C^r$ perturbations of the affine action, and global results concern actions where $C^r$ closeness to the affine action is not assumed (although other restrictions might be present).

We obtain local rigidity results for the actions of $\G_{\bar B} = \G_{\bar B, 1}$, which imply similar results for $\G_{\bar B} = \G_{\bar B, K},\ K\geq 1$. To each action $\alpha$ on $\T^N$ sufficiently $C^r$ close to an affine action for some large $r$, we define an $N\times N$ {\em rotation matrix} $\brho(\alpha)$.   Under suitable hypotheses on $\brho(\alpha)$, if the columns of this matrix satisfy a simultaneous Diophantine condition, then  $\alpha$ is smoothly conjugate to the affine action with rotation matrix $\brho(\alpha)$.  The fact that the action $\alpha$ is a smooth perturbation of an affine action is crucial.

More generally, for each affine action $\bar\alpha$ of  $\G_{\bar B, K}$ there are $K$ rotation matrices  $\brho_1(\bar\alpha), \ldots,\brho_K(\bar \alpha)$.
In the section on global rigidity, we consider actions $\alpha$  of $\G_{\bar B,K}$ for which $\bar B$ acts as an Anosov diffeomorphism, but the rotation matrices $\brho_i(\alpha)$ are not {\em a priori} well-defined.  Under relatively weak additional assumptions on the action, we obtain that the collection of $\brho_i(\alpha)$ can be defined and and forms a complete invariant of the action, up to {\em  topological} conjugacy.  
We then establish conditions under which this topological conjugacy is smooth.
In particular, if $K$ is sufficiently large (depending on the spectrum of $\bar A$ and the Anosov element $\alpha(g_0)$), then for almost every 
set of rotation matrices $\brho_1(\alpha), \ldots,\brho_K(\alpha)$, the conjugacy is smooth.  

Before stating these results, we describe precisely the space of affine actions of $\G_{\bar B, K}$ we consider.


\subsection{The affine actions of $\Gamma_{\bar B, K}$}\label{SSAffine}
The following proposition can be verified directly using the group relation \eqref{EqGrpRelation}. 
\begin{Prop}\label{p=action}
Let $\bar A,\bar B\in  \mathrm{SL}(N,\Z)$, and suppose that  $\brho_1, \brho_2,\ldots, \brho_K $ are  real-valued, $N\times N$ matrices  such that each $\brho = \brho_i$ satisfies:
\begin{equation}\label{EqCommute}
\bar A\brho = \brho \bar B \mathrm{\ mod\ }\ \Z^{N\times N}.
\end{equation}
Denote by $\rho_{i,j}$ the $j$-th column of $\brho_i$.  Then the affine maps
$$\bar\al(g_0)(x) := \bar Ax \,\,(\mathrm{mod } \,\Z^N),\quad \mathrm{and}\quad \bar\al(g_{i,j})(x)  := x +\rho_{i,j} \,\,(\mathrm{mod } \,\Z^N)$$
define an action
 $\bar\al  = \bar\al_K(\bar A, \brho) \colon  \G_{\bar B, K}\to \mathrm{SL}(N,\Z)\ltimes\R^N $ on $\T^N$.

Conversely, if $\alpha\colon  \G_{\bar B,K}\to \mathrm{SL}(N,\Z)\ltimes\R^N $ is an action  on $\T^N$
with $$\alpha(g_0)(x) = \bar Ax \,\,(\mathrm{mod } \,\Z^N),\quad\mathrm{and}\quad\alpha(g_i)(x) =  x + \beta_{i,j} \,\,(\mathrm{mod } \,\Z^N),$$ for some vectors $\beta_{i,j}\in \R^N$, then
for each $i = 1,\ldots, K$, the matrix $\brho_i$ whose columns are formed by the $\beta_{i,j}$ satisfies \eqref{EqCommute}.
\end{Prop}
We further focus on the case $K=1$.  
For $\bar A,\bar B\in  \mathrm{SL}(N,\Z)$  and 
$\brho\in M_N(\T)$, where $M_N(\T)$ denotes $N\times N$ matrices with entries in $\T$, we denote by $\bar\alpha(\bar A, \brho)$ the action $\bar\alpha$ on $\T^N$ defined in Proposition~\ref{p=action}.
Let
$$\mathrm{Aff}(\G_{\bar B},\bar A):=\{\bar\alpha(\bar A,\brho): \brho \mathrm{\ satisfies\ }\eqref{EqCommute}\}.$$
\begin{Prop}\label{PropFaithful}
The action $\al(\bar A, \brho)\in \mathrm{Aff}(\G_{\bar B},\bar A)$ is faithful if and only if  $\bar A$ is not of finite order,  and 
 the column vectors $\rho_1,\ldots,\rho_N$ of $\brho$ are linearly independent over $\Z$; that is, if there exists $(p_1,\ldots,p_N)\in \Z^N$ with $\sum_{i=1}^N p_i\rho_i=0 \mod \Z^N$, then $p_1=\cdots=p_N=0$.
\end{Prop}

For $\bar A \in  \mathrm{SL}(N,\Z)$ not of finite order, we thus define the set of faithful affine actions by
$$\mathrm{Aff}_\star(\G_{\bar B},\bar A):=\{\bar\alpha(\bar A,\brho)\in \mathrm{Aff}(\G_{\bar B},\bar A) :  \bar\alpha(\bar A,\brho) \mathrm{\ is\ faithful}\}.$$ 
\begin{Prop} \label{PropUniqueAffine}
Given $\bar A\in \mathrm{SL}(N,\Z)$, two actions  $\bar\alpha_1, \bar\alpha_2 \in \mathrm{Aff}_\star(\G_{\bar B},\bar A)$ are conjugate by a homeomorphism homotopic to identity if and only if $\bar\alpha_1=\bar\alpha_2.$ 
\end{Prop}
Thus actions in $\mathrm{Aff}_\star(\G_{\bar B},\bar A)$ may not be locally rigid even among affine actions. 
Returning to our original example,  let
\begin{equation}\label{EqEg}
\bar A= \left(
\begin{array}{cc}
2& 1\\
 1& 1 
\end{array}
\right),\quad \bar B=\bar A
\end{equation}
and
$$\G_{\bar A} = \langle g_0, g_1, g_2\ |\  g_0g_1g_0^{-1} = g_1^2g_2,\quad g_0g_2g_0^{-1} = g_1g_2,\quad [g_1,g_2] =1\rangle.$$
Fix $a_0, a_1 \in \T^1$,
and let
 $\brho(a_0, a_1):= \begin{pmatrix} a_0 + 2 a_1 & a_1 \\  a_1 & a_0 + a_1\end{pmatrix}$.
Then  \[\{\bar\alpha(\bar A, \brho(a_0,a_1)): a_0,a_1\in \T^1\}\] defines a $2$-parameter family of non-conjugate actions on $\T^2$; in the next subsection we explain that these are all such affine actions.

Thus the matrix $\brho$ is a complete invariant of the faithful affine representations $\bar\al(\bar A,\brho)$ of $\G_{\bar B}$.
The columns of $\brho$ are rotation vectors of the corresponding translations.  We will show that these rotation vectors, and hence the invariant $\brho$, extend continuously to a neighborhood of the affine representations in such a way that 
$\brho$ gives a complete invariant under smooth conjugacy, 
under the hypotheses that the columns of $\brho$ satisfy a simultaneous Diophantine condition.

Further properties of the affine representations are discussed in Appendix \ref{SAffine}, which also contains the proofs of the results in this section.

\subsection{Local rigidity of $\G_{\bar B,K}$ actions} 

An action $\al\colon  \G\to \mathrm{Diff}^r(M)$
is  {\em $C^{r,k,\ell}$ locally rigid} if any sufficiently $C^k$ small $C^r$ perturbation $\tilde\al$ is $C^\ell$ conjugate to $\al$, i.e.,
there exists a diffeomorphism $h$ of $M$,  $C^\ell$ close to the identity, that conjugates $\tilde\al$ to $\al$: $h\circ\al(g) =
\al(g)\circ h$ for all $g\in \G$.   The paper of Fisher \cite{Fi} contains  background and an excellent overview of the local
rigidity problem for general group actions.

Local rigidity results for solvable group actions are relatively rare.   In \cite{DK},
Damjanovic and Katok proved $C^{\infty,1,\infty}$ local rigidity for $\Z^k\ (k\geq  2)$ (abelian) higher rank partially
hyperbolic actions by toral automorphisms, by introducing a new KAM iterative scheme.  In
\cite{HSW} and \cite{W}, the authors proved local rigidity for higher rank ergodic nilpotent actions
by toral automorphisms on $\T^N$, for any even $N \geq  6$. Burslem and Wilkinson in \cite{BW} studied
the solvable Baumslag-Solitar groups 
\[BS(1, n) := \langle a, b\ |\ aba^{-1} = b^n; n \geq  2\rangle)\]
acting on $\T^1$
and obtained a classification of such actions and a global rigidity result  in the analytic setting.
Asaoka in \cite{A1, A2} studied the local rigidity of the action on $\T^N$ or $\mathbb S^N$ of non-polycyclic
abelian-by-cyclic groups, where the cyclic factor is uniformly expanding.

Unless assumptions are made on the action (or the manifold), solvable group actions are typically not locally rigid but can enjoy a form of partial local rigidity: that is, local rigidity subject to constraints that certain invariants be preserved.
The simplest example occurs in dimension $1$, where the rotation number of a single $C^2$ circle diffeomorphism supplies a complete topological invariant, provided that it is irrational, and  a complete smooth invariant, provided it satisfies a Diophantine condition.  This result extends to actions of higher rank abelian groups on $\T^1$, under a simultaneous Diophantine assumption on the rotation numbers of the generators of the action \cite{M}.  In fact, these results are not just local in nature but apply to all diffeomorphisms of the circle \cite{FK}.

For higher dimensional tori, even local rigidity results of this type are scarce, one problem being the lack of
invariants analogous to the rotation number.  One result in this direction is by
Damjanovic and Fayad \cite{DF},  who proved local rigidity of ergodic affine $\Z^k$ actions on the
torus that have a rank-one factor in their linear part, under certain Diophantine conditions. 


\begin{Def} \label{DefDiop} A collection of vectors $v_1,\ldots, v_m\in \R^N$ is \emph{simultaneously Diophantine},  if there exist $\tau > 0$ and
$C > 0$ such that
\begin{equation}\label{e=dioph} \max_{1\leq i\leq m} |\langle v_i, n\rangle| \geq
\frac{C}{
\|n\|^{\tau}},\quad \ \forall \ n \in \Z^N\setminus\{0\}.
\end{equation}
We denote by $\mathrm{SDC}(C,\tau)$ the set of $(v_1,\ldots,v_m)$ satisfying \eqref{e=dioph}.

\end{Def}
For example, the matrix $\rho \,\mathrm{Id}_N$ is simultaneously Diophantine if $\rho$ is a Diophantine number. It is known that for any for fixed $\tau>N-1$, the simultaneous Diophantine vectors  
\[\bigcup_{C>0} \mathrm{SDC}(C,\tau)\] form a full Lebesgue measure subset of $\T^{N\times m}$ (\cite{P}).  
 
\begin{Def}
Given a homeomorphism $f: \T^N\to \T^N$ homotopic to the identity and preserving a probability measure $\mu$, the vector \begin{equation}\label{EqRot}
\rho_\mu(f) :=\int_{\T^N} (\tilde f(x)- x)\, d\mu,\ \mathrm{mod}\ \Z^N,
\end{equation}
where $\tilde f\colon  \R^N\to \R^N$ is any lift of $f$, is independent of the choice of lift $\tilde f$. We call $\rho_\mu(f)$ the \emph{rotation vector} of $f$ with respect to $\mu$.
\end{Def}

Our main local rigidity result is: 
\begin{Thm}\label{ThmLocal} For any $\bar A,\bar B \in \mathrm{SL}(N,\Z)$ and any $C,\tau>0$, there exist $\eps>0$ and $\ell\in \N$ such that for any $\brho\in\mathrm{SDC}(C,\tau)$ satisfying \eqref{EqCommute} the following holds. Let  $\al: \Gamma_{\bar B}\to \mathrm{Diff}^\infty(\T^N)$ be any representation satisfying
\begin{enumerate}
\item $\al(g_0)$ is homotopic to $\bar\alpha(\bar A,\boldsymbol \rho)(g_0)=\bar A$;
\item  $\max_{1\leq i\leq N}\|\al(g_i)-\bar\al(\bar A,\boldsymbol \rho)(g_i)\|_{C^\ell}<\eps$;
\item there exist $\al(g_i)$-invariant probability measures $\mu_i$,\ $i=1,\ldots,N$, such
that the matrix formed by the rotation vectors
$(\rho_{\mu_1}(\al(g_1)), \ldots,\rho_{\mu_N}(\al(g_N)))$ is equal to $\brho$.
\end{enumerate}
Then there exists a $C^\infty$ diffeomorphism $h$ that is $C^1$ close to identity such that $h\circ\al = \bar\al\circ h$. Moreover,  the measure $\mu= h^{-1}_*\mathrm{Leb}$, where $\mathrm{Leb}$ is Haar measure on $\T^N$, is the unique $\al$-invariant measure and thus satisfies $\rho_\mu(T_i)=\rho_{\mu_i}(T_i) = \rho_i,\ i=1,\ldots,N.$
\end{Thm}
We will prove in Appendix \ref{AppDiop} that the simultaneously Diophantine condition is actually satisfied by a large class of matrices $\brho$ and $\bar A,\bar B$ satisfying \eqref{EqCommute}. One special case is when $\bar A=\bar B\in \mathrm{SL}(N,\Z)$ has simple spectrum, in which case any matrix $\brho\in M_N(\R)$ commuting with $\bar A$ has the form $\brho=\sum_{i=1}^N a_i \bar A^{i-1}$, $a_i\in \R$, $i=1,\ldots,N$. The columns of the matrix $\brho$ are simultaneously Diophantine if the nonvanishing $a_i$'s form a Diophantine vector. 
\begin{Rk}
We remark that the faithfulness $($guaranteed by the Diophantine condition$)$ of
the action is necessary for smooth conjugacy. For instance, consider $\rho= 1/2$ in \eqref{EqEg} and $\al(g_i)=\bar\al(g_i),\ i=1,2$, and for any $\eps>0$,
$$\alpha(g_0)\left[\begin{array}{c}
x\\
y
\end{array}
\right]=\left[\begin{array}{cc}
2&1\\
1&1
\end{array}
\right]\left[\begin{array}{c}
x\\
y
\end{array}
\right]+\eps \left[\begin{array}{c}
\sin(4\pi x)\\
\sin(4\pi x)
\end{array}
\right].$$
One can verify that this gives rise to a $\G_{\bar A}$ action.  We will see in Theorem \ref{ThmTop2} that for sufficiently small $\eps,$ there exists a bi-H\"older conjugacy $h$ satisfying $h\circ\alpha=\bar\alpha\circ h$. However, the conjugacy $h$ is not $C^1$. Indeed, $0$ is a fixed point for both $\alpha(g_0)$ and $\bar A$. The derivative $D_0\alpha(g_0)=\bar A+4\pi\eps \left[\begin{array}{cc}
1&0\\
1&0\end{array}
\right]$ has determinant $1$ but different trace than $\bar A$ for $\epsilon\neq 0$, so it is not conjugate to $\bar A$. \end{Rk}

{\bf Question:} {\it Suppose the action is faithful and close to an algebraic action, is it always possible
to smoothly conjugate the action to an algebraic one?}

\subsection{Global rigidity}
The proof of  the above local rigidity theorem is an application of the KAM techniques for $\Z^N$ actions initiated by Moser \cite{M} in the context of $\Z^N$ actions by circle diffeomorphisms.  The KAM technique is essentially perturbative. It is natural to ask if our solvable group action is rigid in the nonperturbative sense, i.e. whether it is globally rigid. A class of actions of a group, not necessarily close to a algebraic actions, is called \emph{globally rigid} if any action from this class is conjugate to an algebraic one. 
There is a nonperturbative global rigidity theory for circle maps known as \emph{Herman-Yoccoz theory}. For abelian group actions by circle diffeomorphisms, the global version of Moser's theorem was proved by Fayad and Khanin \cite{FK}. These global rigidity results rely on the Denjoy theorem stating that a $C^2$ circle diffeomorphism with irrational rotation number is topologically conjugate to the irrational rotation by the rotation number. 

In the higher dimensional case, there is no corresponding Herman-Yoccoz theory for diffeomorphisms of $\T^N$ isotopic to rotations. 
The reason is that rotation vectors are not well-defined in general. Even when rotation vectors are uniquely defined, they are not the complete invariants for conjugacy analogous to rotation numbers for circle maps. In particular, the obvious analogue of the topological conjugacy given by the Denjoy theorem does not exist for diffeomorphisms of $\T^N,\ N>1.$ 

On the other hand, by a theorem of Franks (Theorem \ref{ThmFranks} below), Anosov diffeomorphisms of $\T^N$ are topologically conjugate to toral automorphisms. A  diffeomorphism $f\colon  M\to M$ is called \emph{Anosov} if there exist constants $C$ and $0<\lambda<1$ and for each $x\in M$ a splitting of the tangent space $T_xM=E^s(x)\oplus E^u(x)$  such that for every $x\in M$, we have
\begin{itemize}
\item $D_xf E^s(x)=E^s(f(x))$ and $D_xf E^u(x)=E^u(f(x))$,
\item $\| D_xf^n v\|\leq C \lambda^n\|v\|$ for $v\in E^s(x)\setminus \{0\}$ and $n\geq 0,$ and \\
$\| D_xf^n v\|\leq C \lambda^{-n}\|v\|$ for $v\in E^u(x)\setminus \{0\}$ and $n\leq 0.$
\end{itemize}
 As the starting point of a global rigidity result of our $\G_{\bar A}$ action, we assume $\alpha(g_0)$ acts by an Anosov diffeomorphism homotopic to $\bar A$. With the topological conjugacy at hand, the next question is to show the topological conjugacy given by Franks's theorem also linearizes the abelian subgroup action. The new problem that arises is that for toral diffeomorphisms homotopic to identity the rotation vector is in general not well-defined, and it only makes sense to talk about the rotation set. When there is more than one vector in the rotation set, the diffeomorphism cannot be conjugate to a translation. 

\subsubsection{Topological conjugacy}


The case $N=2$ admits a fairly complete understanding of the topological picture of $ABC$ actions.  In particular, the next result classifies the ABC group actions on $\T^2$ up to topological conjugacy when $g_0$ acts by an Anosov diffeomorphism and the $g_i$, for $i\geq 1$ are not too far from translations, in a sense that we make precise.

\begin{Thm}\label{ThmTop2}
Let  $\bar A,\bar B\in \mathrm{SL}(2,\Z)$ be linear Anosov and $\al\colon  \G_{\bar B}\to \mathrm{Diff}^r(\T^2),$ $r>1$, be a representation satisfying 
\begin{enumerate}
\item$\al(g_0)$ is Anosov and homotopic to $\bar A$;
\item the sub-action generated by $\al(g_1),\ldots,\al(g_N) $ has  sub-linear oscillation $($see Definition \ref{DefSlowoscillation} below$)$ in the case of tr$\bar A$=tr$\bar B$ and $c$-slow oscillation in the case of tr$\bar A\neq$tr$\bar B$ where $c$ is in Remark \ref{RkcSlow}. 
\end{enumerate}
Then there exist $\boldsymbol\rho$ satisfying \eqref{EqCommute} and a unique bi-H\"older homeomorphism $h\colon  \T^2\to \T^2$ homotopic to the identity satisfying $$h\circ\al=\bar\al(\bar A,\boldsymbol\rho)\circ h.$$
\end{Thm}
\begin{Rk}
In this theorem faithfulness of the action is not necessary, since we do not need the rotation vectors of $\bar\alpha(g_i)$ to be irrational. 
\end{Rk}

The assumption on the sub-linear oscillation is removed in \cite{HX} by introducing an Anosov foliation Tits' alternative. Here we give the statement and refer the readers to \cite{HX} for the proof. 
\begin{Thm}[\cite{HX}] 

Suppose that $\al: \Gamma_{\bar B} \to \Diff(\T^2)$ is such that:

\begin{enumerate}
\item $\bar B \in \mathrm{SL}(2,\Z)$ is an Anosov linear map $($i.e. $\bar B$ has eigenvalues of norm different than one.$)$
\item The diffeomorphism $\al(g_0)$ is  Anosov and homotopic to $\bar A$.
\end{enumerate}

Then $\al$ is topologically conjugate to an affine action of $\Gamma_{\bar B}$ as in Proposition \ref{p=action} up to finite index. More concretely, there exist a finite index subgroup $\Gamma' <\Gamma_{\bar B}$ and  $h \in \mathrm{Homeo}(\T^2)$ such that $h\al(\Gamma') h^{-1}$ is an affine action.

\end{Thm}


With the notion of $c$-slow oscillation, we also obtain the following result for general $N$. 
\begin{Thm}\label{ThmTopN}
Suppose $N>2$. Given hyperbolic matrices $\bar A,\bar B\in \mathrm{SL}(N,\Z)$, there exists $0\leq c<1$ such that the following holds. 
Let $\al\colon  \G_{\bar B}\to \mathrm{Diff}^r(\T^N)$, $r>1$, be a representation satisfying 
\begin{enumerate}
\item $\al(g_0)$ is Anosov and homotopic to $\bar A$,
\item the sub-action generated by $\al(g_1),\ldots,\al(g_N)$ has $c$-slow oscillation $($see Definition \ref{DefSlowoscillation} below$)$. 
\end{enumerate}
 Then there exist $\boldsymbol \rho$ satisfying \eqref{EqCommute} and a unique bi-H\"older homeomorphism $h\colon  \T^N\to\T^N$ homotopic to the identity with $$h\circ\al=\bar\al(\bar A,\boldsymbol \rho)\circ h.$$ 
\end{Thm}
\begin{Rk}\label{RkcSlow}
 The constant $c$ in Theorem \ref{ThmTopN} can be made explicit as follows. Suppose $\bar A$ has eigenvalues $\lb_1^u,\ldots,\lambda^u_k$ and $\lb_1^s,\ldots,\lambda^s_\ell$, $k\geq 1,\ \ell\geq 1$, $k+\ell=N$,  (complex eigenvalues and repeated eigenvalues are allowed), ordered as follows
\begin{equation}|\lambda^s_\ell|\leq \ldots\leq |\lambda_1^s|<1<|\lambda_1^u|\leq \ldots\leq |\lambda_k^u|.\end{equation}
We introduce similar quantities for $\bar B$
\begin{equation}|\mu^s_{\ell'}|\leq \ldots\leq |\mu_1^s|<1<|\mu_1^u|\leq \ldots\leq |\mu_{k'}^u|,\quad \ell'+k'=N.\end{equation}

 Then $c$ can be chosen to be any number satisfying
\begin{equation}\label{EqcSlow}
0\leq c<\min\left\{\frac{\ln |\lambda_1^u|}{\ln|\mu_{k'}^u|}, \frac{\ln |\lambda_1^s|}{\ln|\mu_{\ell'}^s|}\right\}.\end{equation}
\end{Rk}
We next introduce the concept of sublinear deviation and $c$-slow deviation. 
Let $T\in \mathrm{Diff}_0(\T^N)$ and let $\tilde T\colon  \R^N\to \R^N$ be a lift of $T.$ Denote by $\pi_i$ the projection to the $i$-th component of a vector in $\R^N$. Define the \emph{oscillation} $\mathrm{Osc}(\tilde T)$ of $\tilde T$ by: $$\mathrm{Osc}(\tilde T):=\max_{x,i}\{\pi_i(\tilde T(x)-x)\}-\min_{x,i}\{\pi_i(\tilde T(x)-x)\}.$$
It is easy to see that Osc is independent of the choice of the lift. We define Osc$(T)=$Osc$(\tilde T)$.
\begin{Def}\label{DefSlowoscillation} \begin{enumerate} 
\item For given $c\in [0,1)$, we say that the abelian group action $\beta\colon  \Z^N\to \mathrm{Diff}_0^r(\T^N)$ is {\em of $c$-slow oscillation} if $$\limsup_{\|p\|\to\infty}\frac{\mathrm{Osc}(\beta(p))}{\|p\|^c}<\infty.$$ 
\item We say the action $\beta$ is {\em of bounded oscillation} if it is of $0$-slow oscillation. 
\item We say the action $\beta$ has {\em sublinear oscillation} if $$\limsup_{\|p\|\to\infty}\frac{\mathrm{Osc}(\beta(p))}{\|p\|}=0.$$. 
\end{enumerate}
\end{Def}

Let us motivate the definition of $c$-slow oscillation a bit. In the circle map case, the existence and uniqueness of rotation number relies crucially on the fact that the graph in $\R^2$ of every lifted orbit stays within distance $1$ of a straight line, and the rotation number is simply the slope of the line. This fact is also important in the study of Euler class and bounded cohomology for groups acting on circle \cite{Gh}. We say a diffeomorphism $f\colon  \T^N\to\T^N$ is of \emph{bounded deviation} if there exists $\rho\in\T^N$ and a constant $C>0$, such that 
$$\|\tilde f^n(x)-x-n\rho\|_{C^0}\leq C,\quad \forall\ n\in \Z.$$

Being of bounded deviation implies that each orbit of $\tilde f$ stays within bounded distance of the line $\R \rho$. The concept of bounded deviation was first introduced by Morse, who called it of class A, in the case of geodesic flows on surfaces of genus greater than 1 \cite{Mo}. It was later shown by Hedlund that globally minimizing geodesics for an arbitrary smooth metric on $\T^2$ are also of bounded deviation \cite{He}. A generalization to Gromov hyperbolic spaces can be found in \cite{BBI}. In the one-dimensional case, all circle maps are of bounded deviation, from which follows immediately the existence of the rotation number. 

Being of bounded deviation does {\em not} however guarantee the existence of a conjugacy to a rigid translation. In the one dimensional case, a circle map with irrational rotation number 
is only known to be semi-conjugate to a rotation. Denjoy's counter-example shows that the semi-conjugacy cannot be improved to a conjugacy without further assumptions. 
In the two dimensional case, it is known \cite{Ja} that for a conservative pseudo-rotation of bounded oscillation, the rotation vector being totally irrational is equivalent to the existence of a semi-conjugacy to the rigid translation. Examples of diffeomorphisms on $\T^2$ of bounded deviation can be found in \cite{MS}, which are higher dimensional generalizations of Denjoy's examples on $\T^1$. 


It is easy to see that bounded deviation implies $c$-slow oscillation with $c=0$. 
Sublinear oscillation occurs in  first passage percolation (see Section 4.2 of \cite{ADH}) where paths minimizing a cost defined for random walks on $\Z^2$ have $c$-slow oscillation with a power law $c\leq 3/4$ and conjecturally $c=2/3$. 



\subsubsection{Smooth conjugacy}
The conjugacy $h$ in Theorem \ref{ThmTop2} and \ref{ThmTopN} is only known to be H\"older. It is natural to ask if we can improve the regularity. In hyperbolic dynamics, there is a periodic data rigidity theory for Anosov diffeomorphisms, which implies in the two-dimensional case that if the regularity of $h$ is known to be $C^1$, then $h$ is in fact as smooth as the Anosov diffeomorphism $\alpha(g_0)$ (see Theorem \ref{ThmdelaLlave} below). 

So the problem is now to find sufficient conditions for our action to ensure that the conjugacy $h$ is $C^1$. The invariant foliation structure given by the Anosov diffeomorphism enables us to generalize the Herman-Yoccoz theory for circle maps to the higher dimensional setting. 

To obtain higher regularity of the conjugacy, we consider a slightly different class of ABC groups $\G_{\bar B,K}$ for some $K\geq 1.$ 

We introduce the following condition: 

\centerline{$(\star)$ {\it $\brho$ rationally generates $\T^N$,}}
\noindent meaning:
 the set $\left\{\sum_{i=1}^N p_i\brho_i\mathrm{\ mod\ }\Z^N\ |\ (p_1,\ldots,p_N)\in \Z^N\right\}$ is dense in $\T^N$,  where $\brho_i$ denotes the $i$th column of $\brho$.



\begin{Thm}\label{Thm2} 
Let $\bar A,\bar B\in \mathrm{SL}(2,\Z)$ with tr$\bar A$=tr$\bar B$. Given an Anosov diffeomorphism $A\colon  \T^2\to \T^2$ homotopic to $\bar A\in\mathrm{SL}(2,\Z)$, there is a $C^1$ open set $\mathcal O$ of Anosov diffeomorphisms containing $A$, and a number $K_0$ such that for any integer $K\geq K_0$, there exists a full measure set $\mathcal R_{2,K}\subset (\T^2)^K$ such that the following holds. 

Let $\alpha\colon  \G_{\bar B,K}\to \mathrm{Diff}^r(\T^2)$ be a representation satisfying:
\begin{enumerate}
\item $\alpha(g_0)\in \mathcal O$,
\item the sub-action generated by $\al(g_{1,1}),\al(g_{2,1})$ has sub-linear oscillation, 
 and assume in addition that $\boldsymbol \rho$ given by Theorem \ref{ThmTop2} satisfies $(\star)$. 
 \item for some $i\colon  \{1,\ldots,K\}\to \{1,2\}$, the rotation vectors $(\rho_{i(1),1},\ldots, \rho_{i(K),K})$ lie in $\mathcal R_{2,K}$, where $\rho_{j,k}$ is the rotation vector of $\alpha(g_{j,k})$ with respect to an invariant probability measure $\mu_{j,k}$, $j=1,2$ and $k=1,\ldots,K$.
  \end{enumerate}
  Then there exists a unique $C^{r-\eps}$ conjugacy $h$ conjugating the action $\al$ to an affine action for $\eps$ arbitrarily small.
\end{Thm}
\begin{Thm}\label{ThmMain}
Given a hyperbolic $\bar A\in\mathrm{SL}(N,\Z)$, $N>2$, with simple real spectrum, there exist a $C^1$ neighborhood $\mathcal O$ of $\bar A$, a number $0\leq c<1$ and a number $K_0$, such that for any integer $K>K_0$, there exists a full measure set $\mathcal R_{N,K}\subset (\T^N)^K$ such that the following holds. 

Let  $\alpha\colon  \Gamma_{\bar B,K}\to \mathrm{Diff}^r(\T^N)$ be a representation satisfying
\begin{enumerate}
\item $\al(g_0)\in \mathcal O$;
\item the sub-action generated by $\al(g_{1,1}),\ldots, \al(g_{N,1})$ has $c$-slow oscillation and assume in addition that $\brho$ given by Theorem \ref{ThmTopN} satisfies $(\star)$; 
\item  for some $i\colon  \{1,\ldots,K\}\to \{1,\ldots,N\}$, the rotation vectors $(\rho_{i(1),1},\ldots, \rho_{i(K),K})$ lie in $\mathcal R_{N,K}$, 
where $\rho_{j,k}$ is the rotation vector of $\alpha(g_{j,k})$ with respect to an invariant probability measure $\mu_{j,k}$, $j=1,\ldots,N$ and $k=1,\ldots,K$.
\end{enumerate}
Then  there is a unique $C^{1,\nu}$ conjugacy
$h$ conjugating $\al$ to an affine action for some $\nu>0$.
\end{Thm}

In dimension $3$, the regularity of the conjugacy can be improved applying the work of Gogolev in \cite{G2} (Theorem \ref{ThmGogolev} below). 
\begin{Cor}\label{Thm3} Under the same assumptions as Theorem \ref{ThmMain}, suppose in addition that $N=3$ and $r>3$. Then the conjugacy $h\in C^{r-3-\eps},$ for arbitrarily small $\eps.$ Moreover, there exists a $\kappa\in \Z$ such that if $r\notin(\kappa,\kappa+3)$, then $h\in C^{r-\eps}.$
\end{Cor}
Further relaxation of the assumptions of Theorem \ref{ThmMain} and Corollary \ref{Thm3} is possible, snd we discuss this in Section \ref{SSAlternative}. In particular, in many cases the condition on the $C^1$ closeness of $A$ to $\bar A$ can be relaxed.

For the $N>2$ case, the elliptic dynamics techniques in two dimensions carry over completely. However, there are two new obstructions that come from the hyperbolic dynamics. On the one hand, a conjugacy between two Anosov diffeomorphisms sends (un)stable leaves to (un)stable leaves. On the other hand the affine foliations parallel to the {\em eigenspaces} of $\bar A$ might not be sent to $A$-invariant foliations with smooth leaves. Adding to the difficulty is the fact that the regularity of the weakest stable and unstable distributions are low (only H\"older in general). These issues present an obstacle to developing a theory of periodic data rigidity as strong as the two-dimensional setting. The most general result \cite{G1,GKS} in this direction for $N>2$  states that {\em if $A$ and  $\bar A$ are $C^1$ close} and  have the same periodic data, then the conjugacy $h$ is  $C^{1+}$ (i.e. $Dh$ and $Dh^{-1}$ are H\"older).  


The paper is organized as follows. We prove the
local rigidity Theorem \ref{ThmLocal} in Section \ref{SLocal}. All the remaining sections are devoted to the proof
of the global rigidity results. In Section \ref{SConj}, we prove that there is a common conjugacy (Theorem \ref{ThmTop2} and \ref{ThmTopN}). In Section \ref{SEllipHyp}, we prepare techniques from elliptic dynamics and hyperbolic
dynamics. In
Section \ref{SC1}, we state and prove the main propositions needed for the proof of Theorem \ref{Thm2} and \ref{ThmMain}. In Section \ref{SProofs}, we prove the main Theorems \ref{Thm2} and \ref{ThmMain}. 
In Appendix \ref{Appendix}, we give the proof of the number theoretic result Theorem \ref{ThmFedja}. In Section \ref{SAffine}, we prove the results about affine actions stated in Section \ref{SSAffine}. 

\section{Local rigidity: proofs}\label{SLocal}
In this section, we prove Theorem \ref{ThmLocal}. 
Here is a sketch.  
Given representation $\al\colon  \Gamma_{\bar B}\to\mathrm{Diff}^\infty(\T^N)$ with $(\rho_{\mu_1}(\al(g_1)),\ldots,\rho_{\mu_N}(\al(g_N)))=\boldsymbol \rho\in \mathrm{SDC}(C,\tau)$, where $\mu_i$ is a invariant probability measure of $\al(g_i)$, we can proceed as in \cite{M} using the KAM method to show that the abelian subgroup action can be smoothly conjugated to rigid translations. Using the group relation, we can further show that this conjugacy also conjugates the diffeomorphism $\al(g_0)$ to a linear one. 

The following proposition is proved by the standard KAM iteration procedure.
\begin{Prop}[KAM for abelian group actions]\label{LmKAM} 
Given $C,\tau>0$, there exist $\ell\geq 1$
and $\eps_0>0$ such that the following holds. 

Let $T_1,\ldots, T_m\in \mathrm{Diff}_0^\infty(\T^N)$ be commuting diffeomorphisms with $m>1$.
Suppose there exist $T_k$-invariant measures $\mu_k$ such that the rotation vectors $\rho_{\mu_k}(T_k),\ k=1,\ldots,m,$ satisfy the simultaneous Diophantine condition with constants $C,\tau$, and $$\max_{1\leq k\leq m}\|T_k-\mathrm{id}-\rho_{\mu_k}(T_k)\|_{C^\ell} < \eps_0.$$ 
Then there exists a $C^\infty$ diffeomorphism $h$ that is $C^1$ close to
the identity such that $$h\circ T_k(x)= h(x)+\rho_{\mu_k}(T_k),\quad x\in \T^N,\quad k=1,\ldots,m.$$ 
Moreover the invariant measure $\mu = h^{-1}_*\mathrm{Leb}$, where $\mathrm{Leb}$ is Haar measure
on $\T^N$,  satisfies $\rho_\mu(T_k)=\rho_{\mu_k}(T_k)$,\ $k=1,\ldots,m$.\end{Prop}
\begin{proof}
The proof of this lemma is essentially the same as Moser \cite{M}. A proof was sketched by F. Rodriguez-Hertz in the case of $m=2, N = 2$ (see Theorem 6.5 of \cite{R1}). It is not difficult to adapt the proof to the case $N > 2, m\geq 2$. The only complexity in the case $N>1$ is caused by the fact that the rotation vector is in general not uniquely defined. Here we give a sketch of the KAM iteration procedure to explain how to incorporate the rotation vector and invariant measure. 


 Given $T_1,\ldots T_m$, we want to find a conjugacy $h$ as stated. The strategy of the KAM iteration scheme is to find a sequence $\{h^{(n)}\},n\geq 0$ of diffeomorphisms such that $\lim_n \mathscr H_n \to h$ in the $C^1$ topology, where $\mathscr H_n =h^{(n)}\circ\cdots\circ h^{(1)}$. 
This limit $h$ is {\em a priori} only $C^1$, but a standard argument then shows that $h$ is smooth.
Let $h^{(0)} = id$, and for $k=1,\ldots, m$, let  $T_k^{(n)}(x):=\mathscr H_n T_k \mathscr H_n^{-1}(x)$, for $n\geq 0 $.  Let $R_k^{(n)}\colon \T^N\to \R^N$ be defined by the equation $T_k^{(n)}(x) = x+\rho_{\mu_k}(T_k)+R_k^{(n)}(x)$.

We show that  for each $k$,  the sequence $R_{k}^{(n)}$ converges to zero in the $C^1$ topology as $n\to\infty$. When $T_k^{(n-1)}$ is known from the previous step, each $h^{(n)}$ is found by solving the linearization of the equation $h^{(n)}T_k^{(n-1)}= h^{(n)}+\rho_{\mu_k}(T_k)$. Since the equation is not solved exactly in each step, the conjugated map $T_k^{(n)}=h^{(n)}T_k^{(n-1)}(h^{(n)})^{-1}$ is not yet the translation $x\mapsto x+\rho_{\mu_k}(T_k)$ but is closer to it than $T_k^{(n-1)}$ is. The standard KAM method consists mainly of two ideas: the solvability of each linearized equation under the Diophantine condition up to some loss of derivatives, and the convergence of the procedure due to the quadratic smallness of $R_k^{(n+1)}$ compared with $R^{(n)}_k$. The key observation of \cite{M} that will also be important here is that the commutativity enables us to solve for one $h^{(n)}$ simultaneously for all $k=1,\ldots,m$ assuming the SDC. 

{\bf Step 1: the cohomological equation and  commutativity. }

  Write $T_k(x)=x+\rho_k+R_k(x)$ for $k=1,\ldots,m$, where $\rho_k=\rho_{\mu_k}(T_k)$ is the rotation vector of $T_k$ with respect to the given measure $\mu_k$. For the sake of iteration later, we will also label $T_k=T_k^{(0)},R_k=R_k^{(0)}$ and $\mu_k=\mu_k^{(0)}$.  The vector $\rho_k$ will be kept constant independent of the super-script. 

 The conjugacy equation $h T_k=h+\rho_k$ gives
\[x+\rho_k+R^{(0)}_k(x)+ H(x+\rho_k+R_k^{(0)}(x))=x+H(x)+\rho_{k},\ \mathrm{where\ } h(x)=x+H(x),\]
whose linearization is \begin{equation}\label{EqLinearize}
H(x+\rho_k)-H(x)=-R^{(0)}_k(x).\end{equation}
Taking Fourier expansions $H(x)=\sum_{n\in \Z^N} \hat H_n e^{2\pi i \langle n,x\rangle }$ and $R^{(0)}_k(x)=\sum_{n\in \Z^N} \hat R^{(0)}_{k,n} e^{2\pi i \langle n,x\rangle }$, we get for $n\neq 0$ \begin{equation}\label{EqFourierKAM}\hat H_n(e^{2\pi i\langle\rho_k, n\rangle}-1)=-\hat R^{(0)}_{k,n}.\end{equation}
The commutativity condition $T_kT_j=T_jT_k$ gives
$$x+\rho_j+R^{(0)}_j(x)+\rho_k+R^{(0)}_k(x+\rho_j+R^{(0)}_j(x))=x+\rho_k+R^{(0)}_k(x)+\rho_j+R^{(0)}_j(x+\rho_k+R^{(0)}_k(x)),$$
whose linearization is\begin{equation}\label{EqCommLinear}
R^{(0)}_k(x+\rho_j)-R^{(0)}_k=R^{(0)}_j(x+\rho_k)-R^{(0)}_j. \end{equation}
In terms of Fourier coefficients,
\begin{equation}\label{EqCommLinearFourier}
\hat R^{(0)}_{k,n}(e^{i2\pi\langle \rho_j,n\rangle }-1)=\hat R^{(0)}_{j,n}(e^{i2\pi\langle \rho_k,n\rangle }-1),\quad n\in \Z^N\setminus\{0\}. \end{equation}

The key point is that the commutativity equation \eqref{EqCommLinearFourier} implies that the solution of the  cohomological equation \eqref{EqFourierKAM} for some $k$ also solves the same equation for all the other $j\neq k$. 

{\bf Step 2: the Fourier cut-off and solving the cohomological  equation.}

We next show how to solve the cohomological equation \eqref{EqFourierKAM}. By the simultaneous Diophantine condition, there exists $C>0$ such that for each $n\in \Z^{N}\setminus\{0\}$, there exists $k=k(n)\in\{1,\ldots,m\}$ such that $|\langle\rho_{k}, n\rangle|\geq \frac{C}{\|n\|^\tau}$ where $\rho_k=\rho_{\mu_k}(T_k)$ is the rotation vector.

We take a Fourier cutoff so that we can control higher order derivatives via  lower order derivatives. For $J^{(0)}\in \N$, we  solve for $\hat H_n$ with $|n|<J^{(0)}$ so that we have $\|\bar H^{(1)}\|_{C^{\ell+\tau}}\leq C(J^{(0)})^{\tau}\|\bar H^{(1)}\|_{C^{\ell}}$, where $\bar H^{(1)}(x) := \sum_{|n|\leq J^{(0)}} \hat H_n e^{2\pi i \langle n,x\rangle }$.  Solving \eqref{EqFourierKAM} for  $k=k(n)$, we get 
\[\hat H_n=-(e^{2\pi i\langle\rho_{k(n)}, n\rangle}-1)^{-1}\hat R_{k(n),n},\ \mathrm{for\ all\ }|n|\leq J^{(0)}.\] 
Denoting $h^{(1)}(x)=x+\bar H^{(1)}(x)$ , we get the estimate $\|\bar H^{(1)}\|_{C^{\ell}}\leq C\| R^{(0)}\|_{C^{\ell+\tau}}$ by the SDC.

From equation  \eqref{EqFourierKAM} and \eqref{EqCommLinearFourier}, we get that $\bar H^{(1)}$ solves the following equation, for all $k$:
\begin{equation}\label{Eq1stIteration}\bar H^{(1)}(x+\rho_k)-\bar H^{(1)}(x)=-\Pi_{J^{(0)}} R_k^{(0)}+\hat R_{k,0}^{(0)},\end{equation}
where $\Pi_{J^{(0)}}$ denotes the projection to Fourier modes with $|n|<{J^{(0)}}$, and the constant $\hat R_{k,0}^{(0)}$ is the $0$thn Fourier coefficient of  $R_k^{(0)}$.

{\bf Step 3: the iteration.}

Further introduce,  for $k=1,\ldots,m,$
\[T^{(1)}_k=h^{(1)}T^{(0)}_k(h^{(1)})^{-1}=x+\rho_k+R_k^{(1)},\quad \hbox{and }\mu_k^{(1)}=h^{(1)}_*\mu_k^{(0)},\]
where $R_k^{(1)}$ is defined as follows. 
Expanding the expression $h^{(1)}T^{(0)}_k=T^{(1)}_kh^{(1)}$, we get for all $k$:
$$x+ \rho_k+R_k^{(0)}+  \bar H^{(1)}(x+\rho_k+R_k^{(0)})=x+\bar H^{(1)}(x)+\rho_k+R_k^{(1)}\circ h^{(1)}.$$
Comparing with \eqref{Eq1stIteration}, we obtain for all $k$:
 \begin{equation} \label{EqRemainder}R_k^{(1)}=(\bar H^{(1)}(x+\rho_k+R_k^{(0)})-\bar H^{(1)}(x+\rho_k))\circ (h^{(1)})^{-1}+(R_k^{(0)}-\Pi_{J^{(0)}} R_k^{(0)})\circ (h^{(1)})^{-1}+\hat R_{k,0}^{(0)}.\end{equation}

Since the conjugation by $h^{(1)}$ does not change the rotation vector, we have
\[\rho_k=\rho_{\mu_k}(T_k)=\rho_{\mu_k^{(0)}}(T_k^{(0)})=\int_{\T^N} \tilde T^{(0)}x-x\,d\mu^{(0)}_k=\int_{\T^N} \tilde T^{(1)}x-x\,d\mu^{(1)}_k=\rho_{\mu_k^{(1)}}(T_k^{(1)});\]
from the equation $\rho_k=\int \tilde T_k^{(1)}\,d \mu^{(1)}_k$, we get $\int R_k^{(1)}\,d\mu_k^{(1)}=0$, so that the $j$th component $j=1,\ldots,N$ of $R_{k}^{(1)}$ vanishes at some point $x_j$. We see from \eqref{EqRemainder} that $\hat R_{k,0}^{(0)}$ is bounded by the $C^0$ norm of the first two terms on the RHS. The remainder $R_k^{(1)}$ consists of the quadratically small error discarded when deriving \eqref{EqLinearize}, as well as the higher Fourier modes with $|n|\geq J^{(0)}$ in $R_k^{(0)}$. We thus obtain from \eqref{EqRemainder} that   
\begin{equation}\label{EqLoss}\|R_k^{(1)}\|_{C^1}\leq C\|\bar H^{(1)}\|_{C^{2}}\|R_k^{(0)}\|_{C^1}+C\|(R_k^{(0)}-\Pi_{J^{(0)}} R_k^{(0)})\|_{C^{1}}.\end{equation}
We set $\eps^{(1)}=\max_k \|R_k^{(1)}\|_{C^1}$. 

The standard KAM method in \cite{M} then applies by repeating the above procedure for infinitely many steps, during which we shall let $J^{(n)}\to\infty$, $\eps^{(n)}\to 0$. The loss of derivative in \eqref{EqLoss} is handled in the standard way using the quadratic smallness on the RHS of \eqref{EqLoss}. Higher order derivative estimates are obtained by interpolation between the $C^1$ estimate in \eqref{EqLoss} and $C^{\ell+\tau}$ estimate due to the Fourier cut-off for some large $\ell$. In the limit, we get the conjugacy $h$ in the statement. Since a collection of translations satisfying the simultaneous Diophantine condition is uniquely ergodic on the torus, we get the common invariant measure $\mu$ must equal $h^{-1}_*\mathrm{Leb}$. The statement on the rotation vectors follows from
$\rho_{\mu_k}(T_k)=\rho_{h^*\mu_k}(\bar T_k)=\rho_{\mathrm{Leb}}(\bar T_k)=\rho_{\mu}(T_k),$
where $\bar T_k(x)=x+\rho_{\mu_k}(T_k).$

\end{proof}

Now we are ready to prove  local rigidity.
\begin{proof}[Proof of Theorem \ref{ThmLocal} (local rigidity).] 

Let $T_i=\al(g_i)$ and let \[\bar T_i=\bar\al(\bar A,\boldsymbol\rho)(g_i): x\mapsto x+\rho_{\mu_i}(T_i),\,\; i=1,\ldots,N.\]
Using the commutativity of  the $T_j$ and the simultaneous Diophantine condition,  we apply Proposition \ref{LmKAM} to construct $h$ 
that simultaneously conjugates $T_i$ to $\bar T_i$:
\[h\circ T_i = \bar T_i \circ h, \,\; i=1,\ldots,N.\]
 We then
compose with $h^{-1}$ on the right and $h$ on the left on both  sides of the group
relation $AT_i = (\prod_{j=1}^N  T_j^{b_{ji}})A$ 
to get
\[h Ah^{-1} \bar T_i=(\prod_{j=1}^N \bar T_j^{b_{ji}})hAh^{-1};\]
in other words, 
\begin{equation}\label{e=Beqn}
hAh^{-1}(x+\rho_j)=hAh^{-1}(x)+\sum_{j=1}^N b_{ji}\rho_j,\mathrm{\ mod\ }\Z^N,\quad i=1,\ldots,N.
\end{equation}
We introduce the function $F(x) = hAh^{-1}(x)-\bar Ax$ defined from $\T^n$ to $\T^n$.  We can choose a homotopy connecting $h$ to the identity under which $F$ is homotopic to $A-\bar A$.  Since $A$ is homotopic to $\bar A$, the image of $A-\bar A$ is homotopic to a point. Therefore we can treat $F$ as a continuous function from $\T^n$ to $\R^n$.
Combined with \eqref{EqCommute}, equation (\ref{e=Beqn}) then gives
$F(x + \rho_i) = F(x),\  \mathrm{mod}\ \Z^n,\; i=1,\ldots,N.$ Continuity of $F$ implies that $F(x + \rho_i) -  F(x)$ is a constant integer vector. We may choose $n\in \Z$ such that $n\rho_i$ mod $\Z^n$ is arbitrarily close to zero, and so  by the continuity of $F$, this constant integer vector has to be zero.  We thus obtain that $ F(x + \rho_i) = F(x)$.  The Diophantine property of the vectors $\rho_1,\ldots,\rho_N$ implies that the action generated by the $\bar T_i$ on $\T^N$  is ergodic with respect to  $\hbox{Leb}$.  Since the function $F(x)$ is invariant,  there is a  vector $F_0\in \R^N$ such that  $F(x)=hAh^{-1}(x)-\bar Ax=F_0$ almost everywhere (but in fact everywhere, since $F$ is continuous).  

 To kill this constant vector $F_0$, we introduce the translation $t(x) = x + (\mathrm{id}- \bar A)^{-1}F_0$. It is easy to check that $t$ conjugates $\bar A x + F_0$ and $\bar Ax$, i.e. $\bar At(x) + F _0= t(\bar A x)$. Composing the above
$h$ with $t$, we get the conjugacy in the statement of the theorem.\end{proof}

\section{The existence of the common conjugacy}\label{SConj}
In this section, we prove Theorem \ref{ThmTopN}.  We will use  the following result of Franks \cite{Fr}.
\begin{Thm}\label{ThmFranks} If $A : \T^N\to \T^N$ is an Anosov diffeomorphism, then $A$ is topologically
conjugate to a hyperbolic toral automorphism induced by $A_*\colon  H_1(\T^N,\Z)\to H_1(\T^N,\Z)$.\end{Thm}

This result has been generalized to the infranilmanifold case by Manning.
It is also known (\cite{KH}
Theorem 19.1.2) that the conjugacy $h$ is bi-H\"older; i.e. both $h$ and $h^{-1}$ are H\"older continuous. 

\begin{proof}[Proof  of Theorem \ref{ThmTopN}] 

Suppose we are given an action $\al\colon  \G_{\bar B}\to \mathrm{Diff}^r(\T^N)$ such that $\al(g_0)=A$ is Anosov and homotopic to $\bar A$, and $\al(g_i)= T_i$, $i=1,\ldots,N$ has  $c$-slow oscillation, where $c$ satisfies \eqref{EqcSlow}. By Theorem \ref{ThmFranks}, there is a homeomorphism $h$ such that $hAh^{-1} = \bar A$.  Let
$R_i:=h T_ih^{-1}$, for   $i = 1,\ldots, N$. We will show that $R_i(x) = x + \rho_i$ where $\rho_i$ is the rotation vector of $T_i$.  We lift $h$ to $\tilde h\colon  \R^N\to\R^N$ and decompose $$\tilde h(x) = x + g(x),\quad \tilde h^{-1}(x) = x + g_-(x),\quad \tilde T^p(x) = x+\Delta T^p(x),$$
for $p\in \Z^N$ where $g(x)$, $g_-(x)$ and $\Delta T^p(x)$ are $\Z^N$-periodic. 

For $p\in \Z^N$ and $t\in \T^N$, we have 
\begin{equation*}
\begin{aligned}
\tilde R^p(x)& = \tilde h \tilde T^p\tilde h^{-1}(x) \\
&= \tilde T^p\tilde h^{-1}(x) + g(\tilde T^p\tilde h^{-1}(x))\\
&=\tilde h^{-1}(x) + \Delta T^p(\tilde h^{-1}(x)) + g(\tilde T^p\tilde h^{-1}(x))\\
&= x +\Delta T^p(\tilde h^{-1}(x)) + g_-(x) + g(\tilde T^p\tilde h^{-1}(x)).
\end{aligned}
\end{equation*}
Since both $g_-$ and $g$ are uniformly bounded,  it follows that if $\{T^p\}$ has $c$-slow oscillation, then so does $\{R^p\}$.
From the group relation, we obtain for $p\in \Z^N$ and $n\in \Z$,
\begin{equation}\label{EqGrpRelation}
\bar A^n\tilde R^p(x) = \tilde R^{(\bar B^t)^np} \bar A^n(x)+Q_{p,n},\quad \bar A^n(\tilde R^p(x)- x) = (\tilde R^{(\bar B^t)^np} - \id)\bar A^n(x)+Q_{p,n},
\end{equation} 
where $Q_{p,n}$ is an integer vector in $\Z^N$ depending on $p,n$ and the choice of the lifts. 

For each $\tilde R^p,\ p \in \Z^N$, we take the Fourier expansion $\tilde R^p(x)- x =
\sum_{k\in \Z^N} \hat R_k(p)e^{2\pi i\langle k,x\rangle}$, where the
coefficient for $k\neq 0$ is  
$$\hat R_k(p) =\int_{\T^N} (\tilde R^p(x)-x)e^{-2\pi i\langle k,x\rangle}\, dx. $$
The condition that $\{R^p\}$ has $c$-slow oscillation implies that there exist $C, P$ such that when $\|p\| \geq P$,
we have $\|\hat R_k(p)\|\leq C\| p\|^c$, uniformly for all $k\neq 0$. From equation \eqref{EqGrpRelation} we obtain that for all $k\in \Z^N\setminus\{0\}$,
\begin{equation}\label{e=hatR} \hat R_k(p) = \bar A^{-n}\hat R_{(\bar A^t)^{-n}k}((\bar B^t)^np).
\end{equation}
We next consider the splitting of $\R^N$ into $\bar\cW^u(0)\oplus \bar \cW^s(0)$, the direct sum decomposition into unstable
and stable eigenspaces of $\bar A$.  Each $\hat R_k(p)$ is a vector, so we write $\hat R_k(p)=(\hat R_k(p))^u+(\hat R_k(p))^s$ where $(\hat R_k(p))^{u,s}\in \bar \cW^{u,s}(0)$. Applying $\bar A^{n}$ we get $\bar A^{n}\hat R_k(p)=\bar A^{n}(\hat R_k(p))^u+\bar A^{n}(\hat R_k(p))^s$ with the estimate $\|\bar A^{n}(\hat R_k(p))^u\|\geq |\lambda_1^u|^n\|(\hat R_k(p))^u\|. $ Doing this decomposition to the equation \eqref{e=hatR}, we obtain the following estimate for $\|(\bar B^t)^np\| \geq P$:
$$\|(\hat R_k(p))^u\|\leq \frac{1}{|\lambda_1^u|^{n}}
\|\hat R_{(\bar A^t)^{-n}k}((\bar B^t)^{n}p)\|\leq C\frac{\|(\bar B^t)^np\|^c}{|\lambda_1^u|^{n}}\leq C\|p\|^c\left(\frac{|\mu^u_{k'}|^c}{|\lambda_1^u|}\right)^n\to 0$$
as $n\to\infty$, if $c< \frac{\ln|\lambda_1^u|}{\ln|\mu_{k'}^u|}$.
Similarly, letting $n\to-\infty$ and projecting to the $\bar \cW^s(0)$ in the above argument, we get that the
projection of $\hat R_k(p)$ to $\bar \cW^s(0)$ is also $0$. Therefore $\hat R_k(p) = 0$ for all $k\neq 0$.
This implies that each $R^p(x)- x,\ p \in \Z^N$, is a constant. Since a conjugacy does not change the
rotation vector, we have $R_i(x) = x + \rho_i$, where $\rho_i$ the rotation vector of $T_i$, $i=1,\ldots,N$. Next we have $R^p(x) = x + \brho p$,\ $p\in \Z^N$.
This completes the proof. 
\end{proof}

\section{Preliminaries: elliptic and hyperbolic dynamics}\label{SEllipHyp}

In this section, we explain and develop techniques from elliptic  and hyperbolic dynamics that we will use to prove our main results. We first introduce the framework of Herman-Yoccoz-Katznelson-Ornstein for obtaining regularity
of the conjugacy of circle maps and generalize it to abelian group actions on $\T^N$. Next, we
state facts about Anosov diffeomorphisms, including the invariant foliation structure and its regularity properties.
\subsection{Elliptic dynamics: the framework of Herman-Yoccoz-Katznelson-Ornstein} \label{SElliptic}
In this section, we generalize to abelian group actions the framework of Herman-Yoccoz theory for circle maps after Katznelson-Ornstein.
\begin{Def}\label{DefCkHk}  Let $\cF$ be  a continuous foliation of $\T^N$ by one-dimensional uniformly $C^1$ leaves $\cF(x),\ x\in\T^N$, and let $k\geq 1$.
\begin{enumerate}
\item  We denote by
$\cH^k = \cH^k(\T^N)$ the group of $C^k$ diffeomorphisms on $\T^N$, $k\in \N$. 
\item We denote by $\cH^k_\cF$ the subgroup of diffeomorphisms in $\cH^k$ preserving the foliation $\cF$; i.e., $f \cF(x)=\cF(f(x)),\ \forall\ f\in \cH^k_\cF$ and $\forall\ x\in \T^N$.
\item The $C^k$ norm $\|\cdot \|_{C^k(\cF)}$ on $C^k(\T^N, \R^{N})$  along the foliation $\cF$ is defined as follows.  For $\varphi\in C^k(\T^N, \R^{N})$, let 
$$\|\varphi\|_{C^k(\cF)}:=\sum_{i=0}^k\sup_x\|(D_x^i \left(\varphi\vert_{\cF}\right) \|,$$ 
where the norm inside the summand on the right hand side is the operator norm induced by the Euclidean metric restricted to the leaves of $\cF$.
\end{enumerate}
\end{Def}
\subsubsection{Generalization of the framework of Herman after Katznelson-Ornstein}
The following statement about circle maps was known to Herman \cite{H}: 

{\it Suppose that $f\in \cH^k(\T^1)$ takes the form $f=h^{-1}(h(x)+\rho)$, where $h\colon  \T^1\to\T^1$ is a homeomorphism and $\rho\notin\mathbb Q$. Then $h \in \cH^k$ if and only if
the iterates $\{ f^j\}_{j\in\Z}$ are uniformly bounded in $\cH^k$.}

 Following   Katznelson-Ornstein \cite{KO}, we generalize this statement to  abelian subgroup actions. 
 
 \begin{Def}
A collection of $m$ vectors $\rho_1,\ldots,\rho_m\in \T^N$ is said to \emph{rationally generate} $\T^N$ if 
$\{\sum_{i=1}^m p_i\rho_i,\ p_i\in \Z,\ i=1,\ldots,m\}$ is dense on $\T^N$. 
 \end{Def}
\begin{Prop}\label{PropKO} Suppose that for some $k>0$, the maps $T_i\in\cH^k(\T^N),\ i=1,\ldots,m$, commute. Suppose also that there exists a homeomorphism $h: \ \T^N\to \T^N$ such that $T_i = h^{-1}\bar T_ih$, where
$\bar T_i(x) = x + \rho_i, \mathrm{mod\ } \Z^N,\ i = 1, 2,\ldots,m$, and $\rho_1,\ldots,\rho_m$ rationally generate $\T^N$.  Fix a lift  $\tilde h\colon \T^N\to \R^N$  of $h$.

Then the following equality holds for all $x$:
\begin{equation}\label{EqBirkhoff}\tilde h(x)= \mathrm{const}. + \lim_{n\to\infty}
\frac{1}{(2n + 1)^{N}}
\sum_{\|p\|_{\ell^\infty}\leq n}\left(\tilde{T}^p(x)-\brho p\right).
\end{equation}
where $p = (p_1,\ldots, p_m)\in \Z^m$, $\brho = (\rho_1,\ldots,\rho_m)\in \T^{N\times m}$ and $\tilde T^p$ is the lift of $T^p$ satisfying $\tilde T^p(0)=\tilde h^{-1}(\tilde h(0)+\brho p)$.
\end{Prop}
\begin{proof}[Proof of Proposition \ref{PropKO}]
 From $T^p = h^{-1}\bar T^{p}h$, we get $ T^p =  h^{-1}( h(x) + \brho p$). We next fix the lift  $\tilde T^p$ of $T^p$ that satisfies $\tilde T^p(0)=\tilde h^{-1}(\tilde h(0)+\brho p)$ to obtain
$$\tilde T^p(x)-\brho p-\tilde h(x) = (\tilde h^{-1}-\mathrm{id})\circ (\tilde h(x) + \brho p).$$
Averaging over all $p\in\Z^m$ with $\|p\|_{\ell^\infty}\leq n$, and letting $n\to\infty$, we get
$$\tilde h(x)= -\int_{\T^N}(\tilde h^{-1}(x)- x) dx + \lim_{n\to\infty}\frac{1}{
(2n + 1)^{N}}
\sum_{\|p\|_{\ell^\infty}\leq n}\left(\tilde{T}^p(x)-\brho p\right),$$
where to get the integral, we use the fact that the affine action of $\Z^m$ via the rigid translations
$\bar T_i,\ i=1,\ldots,m$ is ergodic with respect to Lebesgue, combined with a version of the Birkhoff ergodic theorem for abelian group
actions (c.f. Theorem 1.1. of \cite{L}). 
\end{proof}
\begin{Cor} \label{CorKO} 
Let the abelian group $\mathcal A=\{T^p\colon  p\in \Z^m\}\ (<\cH^k(\T^N))$ and the conjugacy $h$ be as in Proposition \ref{PropKO}.
\begin{enumerate}
\item Let $\bar\cF = \{\bar\cF(x),\ x\in \T^N\}$ be an affine foliation of $\T^N$ by parallel lines. Let $\cF$ be the (topological) foliation of $\T^N$ whose leaves are $\cF(x) = h^{-1}(\bar\cF(h(x)))$, $x\in \T^N$. 
\item Assume the leaves $\cF(x)$ of the foliation $\cF$ are uniformly $C^1$. Note that this implies that $\cA< \cH^k_\cF(\T^N)$. 
\end{enumerate}
 If the set $\{T^p(x)-x\colon  p\in \Z^m\} \subset C^k(\T^N,\R^N)$ is
precompact in the $\|\cdot\|_{C^k(\cF)}$ norm, then 
$h$ is uniformly $C^k$ along the leaves of $\cF$. Moreover, in the case of $k=1$, we also have that $h^{-1}$ is uniformly $C^1$ along the leaves of $\bar \cF$.

\end{Cor}
The proof of Corollary \ref{CorKO} is given in Section \ref{SSSCor}.

Given a continuous increasing function $\psi(x)\colon  \R_{\geq0}\to\R_{\geq0}$ with $\psi(0)= 0$, we say that a function $f: (X,d)\to (X',d')$ between two metric spaces has {\it modulus of continuity $\psi$} at a point $x_0\in X$, if there exists a constant $C>0$ such that $$d'(f(x_0),f(y))\leq C\psi(d(x_0,y)),$$
for any $y \in X$ sufficiently close to $x_0$. 
\begin{Prop} \label{PropMoC} Let the abelian group $\cA$, the conjugacy $h$, and the foliations $\cF$, $\bar\cF$ 
be as in Corollary \ref{CorKO}. 
 Assume that
$\{T^p(x)-x\colon  p\in \Z^m\} \subset C^1(\T^N,\R^N)$ is
uniformly bounded in the $\|\cdot\|_{C^1(\cF)}$ norm 
and that the mapping $$\brho p\ (\mathrm{mod\ } \Z^N) \mapsto \|T^p(x)-x\|_{C^1(\cF)}$$ has modulus of continuity $\psi$ at $\brho p=0$. Then both $\|D\left(h\vert_{\cF}\right)\|$ and $\|D\left(h^{-1}\vert_{\bar\cF}\right)\|$ have modulus of continuity $\psi$ with respect to the Euclidean metric. 
\end{Prop}
The proof of Proposition \ref{PropMoC} is given in Section \ref{SSSMoC}. 
\subsubsection{Proof of Corollary \ref{CorKO}}\label{SSSCor}
We only prove the case of $k=1$. A similar argument gives the continuity of higher derivatives.


We denote the $n$th Birkhoff average on the right hand side of of \eqref{EqBirkhoff} by $\tilde S_n$, so \eqref{EqBirkhoff} can be rephrased as $\tilde h=\lim_{n\to\infty} \tilde S_n$ up to an additive constant. Since $\cA$ is assumed to be pre-compact in $\cH_\cF^1$ and the pointwise convergence is given by \eqref{EqBirkhoff}, we have that $\{D\tilde S_n|_\cF: n\geq1\}$ is precompact in the $C^0$ operator norm, by Theorem 5.35 of \cite{AB}, which states that the convex hull of compact sets is compact in a completely
metrizable locally convex space. This shows that $h$ is differentiable along $\cF$ and any subsequential limit of $\{D\tilde S_n|_\cF: n\geq1\}$ is $Dh|_\cF$. 


We have proved that $Dh|_\cF$ is continuous. To show that $Dh^{-1}|_{\bar\cF}$ is also continuous, by the implicit function theorem, it is enough to show that $\|D_xh|_{\cF}\|$ is bounded away from zero.  Fix $\eps > 0$ such that $\|D_{x_0}h|_\cF\|<\eps$ for some  $x_0$. Then the same inequality holds in a small neighborhood $B(x_0)$ of $x_0$. By the ergodicity of $T^p$, there exist finitely many $p_i$, $i=1,\ldots,n$ such that $\cup_{i=1}^nT^{p_i}(B(x_0))=\T^N$.  This, combined with the equation
\[D_{T^p(x)}h |_{\cF}\cdot D_xT^p|_\cF=D_xh|_\cF,\]
implies that there exists a constant $C$ independent of $p$ such that 
\begin{equation}\label{e=Dhnormbound2}
\|D_xh|_\cF\|<C\eps, 
\end{equation}
for all $x\in \T^N$

Consider a leaf $\cF(x)$ and $x'\in \cF(x)$. We lift the leaf to the universal cover and consider the image of the segment between $x$ and $x'$ under $\tilde h$, i.e. the line segment between $\tilde h(x)$ and $\tilde h(x')$. Since $\tilde h(x)=x+g(x)$ where $g(x)$ is $\Z^N$ periodic, choosing $x$ and $x'$ far apart on $\cF(x)$ we can make $\|\tilde h(x')-\tilde h(x)\|\geq 1.$ 
Fix such a choice of $x,x'$.  There is a $C^1$ curve $\gamma_{x,x'}\subset \cF(x)$ connecting $x$ to $x'$ with length bounded by a constant $C_{x,x'}$. The image $h(\gamma_{x,x'})$ is a $C^1$ curve connecting $h(x)$ and $h(x')$ with length larger than 1.  
Inequality (\ref{e=Dhnormbound2}) implies that
\[1\leq \|\tilde h(x)-\tilde h(x')\|\leq \int_{\gamma_{x,x'}}\|Dh|_\cF\|\leq \eps C_{x,x'}.\]
This implies that $\epsilon > C_{x,x'}^{-1}$, and so $\|Dh|_\cF\|$ is uniformly bounded below, completing the proof.
\qed
\subsubsection{Proof of Proposition \ref{PropMoC}}\label{SSSMoC}
We first introduce a translation-invariant distance $d_\cF^1$ on $\mathcal A$ that is equivalent to the $C^1$ norm as follows (c.f. \cite{K}). Let $B_1$ be the set of $\phi\in C(\T^n,\R)$ with $\|\phi\|_{C^1(\cF)}=1$.  For $f,g\in \cH^k_\cF$, we introduce 
$d^1_\cF(f,g):=d^1_\cF(f g^{-1},\mathrm{id})$ and $$d^1_\cF(f):=\log(\max\{\Phi(f),\Phi(f^{-1}\})+\sup_x\|f(x)-x\|,$$
where $\Phi(f):=\sup_{\phi\in B_1}\|\phi\circ f\|_{C^1(\cF)}.$ To verify the triangle inequality, we note that \begin{equation}
\begin{aligned}\Phi(fg)&=\sup_{\phi\in B_1}\|\phi\circ (fg)\|_{C^1(\cF)}=\sup_{\phi\in B_1}\left\|\frac{1}{\|\phi\circ f\|_{C^1(\cF)}}\phi\circ (fg)\right\|_{C^1(\cF)}\cdot \|\phi\circ f\|_{C^1(\cF)}\\
&\leq \sup_{\psi\in B_1} \|\psi\circ g\|_{C^1(\cF)}\sup_{\phi\in B_1} \|\phi\circ f\|_{C^1(\cF)}=\Phi(f)\Phi(g).\end{aligned}\end{equation}
The chain rule implies that  $d_\cF^1$ is equivalent to the   $C^1$  distance $\sup_x \|f(x)-x\|_{C^1(\cF)}$.

By the assumption on the  modulus of continuity $\psi$, the map from $\brho p\in \T^N$ to $C^1(\cF)$ via $\brho p\mapsto T^p$ is continuous in the $C^1$ norm at the point $\brho p=0$. By the translation invariance of the $d^1_\cF$ norm, it is continuous at every point $\brho p$ mod $\Z^N$.  From the compactness of $\T^N$ we obtain that $\{ T^p(x),\ p\in \Z^m\}$ is pre-compact in $\cH^1_\cF$.  If then follows from Corollary \ref{CorKO} that the functions $Dh|_\cF$ and $Dh^{-1}|_{\bar\cF}$ are continuous. Differentiating the expression $T^p(x)= h^{-1}( h(x)+\brho p)$ along the leaf $\cF(x)$, we get that
$$D_xT^p|_\cF-\mathrm{id}|_\cF =\left(D_{(h(x) + \brho p)}h^{-1}|_{\bar\cF}-D_{h(x)}h^{-1}|_{\bar\cF}\right) \cdot D_xh|_{\cF}.$$
Since the LHS satisfies the modulus of continuity $\psi$ by assumption, i.e. 
$$\|D_xT^p|_\cF-\mathrm{id}|_\cF\|_{C^0}\leq C \psi(\|\brho p\|),$$ for all $p$ with $\|\brho p\|$ small, we get 
$$\|D_{(h(x) + \brho p)}h^{-1}|_{\bar\cF}-D_{h(x)}h^{-1}|_{\bar\cF}  \|_{C^0}\leq C \psi(\|\brho p\|)(\min_x\|D_xh|_\cF\|)^{-1}.$$
Hence $Dh^{-1}|_{\bar\cF}$ has modulus of continuity $\psi$. 
To get the same modulus of continuity for $Dh|_{\cF}$, we use $Dh|_{\cF}\cdot Dh^{-1}|_{\bar\cF}=\mathrm{id}|_{\bar \cF}$.
\qed

\subsection{Hyperbolic dynamics: invariant foliations of Anosov diffeomorphisms}\label{SSHyperbolic}
In this section, we recall some results from hyperbolic dynamics. 
Our statements concern the the unstable objects  $\cW^u$ and $E^u$; the stable analogues
also hold.
\begin{Def}\label{DefMatherSpec}
A $C^1$ diffeomorphism $A\colon  \T^N\to\T^N$ is a Anosov diffeomorphism with \emph{simple Mather spectrum} if there exists a $DA$-invariant splitting of the tangent space
$$T_x\T^N=E^s_\ell(x)\oplus\ldots\oplus  E^s_1(x)\oplus E^u_1(x)\oplus\ldots\oplus E^u_k(x),\quad k+\ell=N,\ k,\ell\geq 1$$
and numbers $$\underline\mu_\ell^s\leq  \bar\mu_\ell^s<\ldots<\underline\mu_1^s\leq  \bar\mu_1^s<1<\underline\mu_1^u\leq  \bar\mu_1^u<\ldots<\underline\mu_k^s\leq  \bar\mu_k^s $$
such that for some constant $C>1$, $$\frac{1}{C}(\underline\mu_i^{u,s})^n\leq \frac{\|DA^n v\|}{\|v\|}\leq C(\bar \mu_i^{u,s})^n,\quad \forall v\in E^{u,s}_i\setminus \{0\},$$
where $i=1,\ldots,\ell$ for $s$ and $i=1,\ldots,k$ for $u.$
\end{Def}
The next result is classical (see \cite{HPS}). 
\begin{Prop} For any $C^r,\ r > 1$ Anosov diffeomorphism $A : \T^N\to \T^N$ with simple Mather spectrum, the strong invariant distribution $E^u_{
i\leq} := E^u_{i}\oplus \ldots\oplus E^u_{k}$ is uniquely integrable, tangent to a foliation
$\cW^u_{i\leq }$ of $\T^N$ whose leaf $\cW^u_{i\leq}(x)$ passing through $x$ is $C^r, x \in \T^N$. This gives rise to a flag of strong unstable foliations
$$\cW^u_k(x)\subset \cW^{u}_{(k-1)\leq}(x) \subset \ldots \subset \cW^u_{2\leq}(x)
\subset \cW^u_{1\leq}(x):= \cW^u(x),\quad x\in\T^N,$$
where each of the inclusions is proper and $\cW^u_{i\leq}$ sub-foliates $\cW^u_{(i-1)\leq}$ with $C^r$ leaves for $i =
2,\ldots, k$.\end{Prop}
It is known that simple Mather spectrum is an open property in the $C^1$ topology. In particular, if $\bar A$ is a toral automorphism with simple real spectrum, then an Anosov diffeormophism that is $C^1$ close to $\bar A$ has simple Mather spectrum. 

\begin{Prop}[H\"older regularity of the invariant distribution, Theorem 19.1.6 of \cite{KH}] 
For each $i$, the  distribution $E^u_i (x)$
is H\"older in the base point $x$.

The Holder exponent depends only on the expansion and contraction rates $\bar\mu_i^{u,s}$ and $\underline\mu_i^{u,s}$.

\end{Prop}

We  denote the weak unstable bundles for $A$ by 
$E^u_{\leq i}:=E^u_1(x)\oplus E^u_2(x)\oplus\cdots\oplus E^u_i(x),$ and that of $\bar A$ by $\bar E^u_{\leq i}(x):=\bar E^u_1(x)\oplus \bar E^u_2(x)\oplus\cdots\oplus \bar E^u_i(x),$ $i = 1,\ldots,k$, $x\in \T^N$. Denote the unstable foliation of $A$ by $\cW^u$ and that of $\bar A$ by $\bar \cW^u$.

\begin{Prop}[Lemma 6.1-6.3 of \cite{G1}]\label{PropGogolev} Consider $A$ a $C^r$ Anosov diffeomorphism that is $C^1$ close to a linear toral automorphism $\bar A$ with simple real spectrum, and the bi-H\"older
conjugacy $h$ given by Theorem \ref{ThmFranks} with $h \circ A = \bar A \circ h$. Then
\begin{enumerate}
\item $h$ preserves the unstable foliation: $h(\cW^u(x)) = \bar \cW^u(h(x))$, for all $x\in \T^N$;
\item each weak unstable distribution $E^u_{\leq i}$ is uniquely integrable, tangent to a  foliation $\cW^u_{\leq i}$ of $\T^N$,
whose leaf $\cW^u_{\leq i}(x)$ passing through $x\in \T^N$ is $C^{1+}$;
\item each distribution $E^u_{i,j}:= E^u
_{\geq i} \cap E^u_{\leq j},\  i \leq j$, is uniquely integrable, tangent to a foliation with $C^{1+}$ leaves;
\item$h$ preserves the weak unstable foliations:  $h(\cW^u_{\leq i}(x)) = \bar \cW^u_{\leq i}(h(x))$,  for $i = 1,\ldots, k$ and all $x\in \T^N$.
\end{enumerate}
\end{Prop}
We remark that the item (1) does not require any $C^1$ closeness of $A$ to $\bar A$, and it holds under the same assumption as Theorem \ref{ThmFranks}. 

From Proposition~\ref{PropGogolev}  we obtain a flag of weak foliations
$$\cW^u_1 \subset \cW^u_{\leq 2} \subset \ldots\subset \cW^u_{\leq k-1}\subset \cW^u_{\leq k} := \cW^u,$$
where each of the inclusions is proper and $\cW^u_{\leq (i-1)}$ sub-foliates $\cW^u_{\leq i}$ with $C^{1+}$ leaves for $i =
2, \ldots, k$. This flag is preserved by the conjugacy $h$. 

When the weak distributions $E^u_i$ are known to be uniquely integrable, we have the following
proposition, which is proved by  a standard graph transform technique.

\begin{Prop}\label{PropHPS1}
\begin{enumerate}
\item Each weak unstable leaf $\cW^u_{\leq i}$ in item $(2)$ of Proposition \ref{PropGogolev} is  subfoliated
by $\cW^u_i$, whose leaves are uniformly $C^r$.
\item The weakest unstable leaf $\cW^u_1(x)$ is $C^{1+}$, and its tangent distribution $E_1^u(x)$ is H\"older.
\end{enumerate}
\end{Prop}

\section{Elliptic regularity with the help of an invariant distribution}\label{SC1}
In this section, we show how to get regularity of the conjugacy $h$ by combining  elliptic dynamics within invariant distributions with hyperbolic dynamics.

\subsection{Elliptic dynamics within invariant distributions}
We start with a few preparatory lemmas.
\begin{Lm}\label{LmInvariant}
Let $\cF$ be an orientable foliation of $\T^N$ with uniformly $C^1$ leaves.
Suppose a homeomorphism $h$ conjugates the abelian group $\{T^p,\ p\in\Z^m\}<\Diff_0(\T^N)$ to translations $\{\bar T^p(x)=x+\brho p,\ p\in\Z^m\}$ and sends the foliation $\cF$ to  an affine foliation $\bar\cF$.

Denote by $E(x)$ the one-dimensional distribution that is
tangent to the leaf $\cF(x)$. Then the distribution $E(x)$ is invariant under the $DT^p$; that is,
\[D_xT^p\left(E(x)\right) = E(T^p(x)),\]
for all  $p\in\Z^m$ and $x\in \T^N$.
\end{Lm}
\begin{proof} 
The straight line foliation $\bar\cF$ is invariant under translations, so after the conjugation the foliation $\cF$ is also invariant under $\{T^p,\ p\in\Z^m\}.$
The lemma follows directly by differentiating the equation $T^p \cF(x) = \cF(T^p(x))$ along the leaves.
\end{proof}
\begin{Lm}\label{LmLyap}
Suppose the conjugacy $h$ in the previous Lemma \ref{LmInvariant} is bi-H\"older. Then for each
$p\in \Z^m\setminus\{0\}$, all the Lyapunov exponents of $T^p$ with respect to any  invariant probability measure are zero.
\end{Lm}
\begin{proof} Suppose there is an invariant measure $\mu$ with at least one nonzero exponent. Without loss of generality, assume that this exponent is negative.  Pesin theory implies that through  $\mu$-a.e.~$x$, there are  local stable manifolds, which are smoothly embedded disks on which $T^n$ contracts distances at an exponential rate.  Thus for two points $y, y'$ on the same local stable manifold of a point $x$,  we have $\|T^{np}(y)- T^{np} (y')\|$ converges to zero
exponentially fast. 

On the other hand, using the bi-H\"older conjugacy $h$ we have
$$\|T^{np}(y)- T^{np}(y')\| = \|h^{-1}(n\brho p + h(y)) - h^{-1}(n\brho p + h(y'))\|\geq \mathrm{const.}\|h(y)- h(y')\|^\eta.$$
where $\eta$ is the H\"older exponent of $h^{-1}$, which gives a contradiction.\end{proof}

\subsection{A quantitative Kronecker theorem}
We will need the following number theoretic result, whose proof is postponed to the Appendix. 

\begin{Thm}\label{ThmFedja} Let $N,K\in \N$ be given. Then there exists a full measure set $\mathcal O$ in the set $\mathcal M_{N\times K}(\T)$ of matrices of $N\times K$ such that for all $M\in \mathcal O$, the following holds.  

For any small $\epsilon>0$, there exists a constant $C$ such that for any $y \in \T^N$ and any $n\in \N$ there
exist  $q\in \Z^N$, $p\in  \Z^K$ satisfying $\|p\| < n$, such that
the following inequality holds
$$\|Mp-q-y\|\leq C n^{-\frac{K}{N}-\epsilon}.$$
\end{Thm}
This theorem is a quantitative version of the classical Kronecker's approximation theorem. When $K=1$, this is the classical Dirichlet's simultaneous Diophantine approximation theorem where we can set $\epsilon=0$. The $N=1$ case was proved in \cite{K}.

This theorem inspires the following definition.
\begin{Def}\label{DefDim} Suppose the $m$ vectors $\rho_1,\ldots,\rho_m\in  \T^N$ rationally generate $\T^N$, and consider the set
of finite linear combinations
\begin{equation}\label{EqS}
S(\rho_1,\ldots,\rho_m):=\left\{\sum_{i=1}^m p_i\rho_i\ \mathrm{mod\ }\Z^N\ |\ p_i\in\Z,\quad i=1,\ldots,m\right\}.
\end{equation}
For each element $\gamma\in S$, we denote by $\|\gamma\|_w$ the word length $\|\gamma\|_w:=\|p\|_{\ell_1},$ where $p = (p_1,\ldots, p_m)\in
\Z^m$ and by $\|\gamma\|$ the closest Euclidean distance of $\gamma$ $($mod $\Z^N)$ to zero. 

We say that $S$ has \emph{dimension} $d$ if there exists a constant $c$ such that for any $x\in \T^N$ and any $\ell>0$
there exists a point 
$\gamma\in S$ satisfying
$$\|\gamma\|_w\leq\ell,\quad \|\gamma-x\|\leq c\ell^{-d}.$$
\end{Def}
Theorem \ref{ThmFedja} implies that for almost every choice of vector tuple $\rho_1,\ldots, \rho_m \in \T^N$, the
set $S$ formed by linear combinations as above has dimension $m/N -\epsilon$ for all $\epsilon > 0$ small.
\subsection{Organization of the proofs of Theorem \ref{Thm2} and Theorem \ref{ThmMain} }
To prove Theorems \ref{Thm2} and \ref{ThmMain}, we just  need to improve the regularity of the  conjugacies obtained in Theorems \ref{ThmTop2} and \ref{ThmTopN}, respectively.  We carry this out in the following propositions. 

The first proposition chooses the $K_0$ in Theorems \ref{Thm2} and  \ref{ThmMain}. 
\begin{Prop}\label{PropK0}
Given $\eta\in (0,1)$ and $d>2/\eta^2$, there exists $K_0$ such that the following holds: for all $K>K_0$, there exists a full measure set $\mathcal R_{N,K}\subset (\T^N)^K$ such that the set $S$ generated by any tuple of vectors $(\rho_{1},\ldots, \rho_{K})$ lying in $\mathcal R_{N,K}$ is dense on $\T^N$ and has dimension $d$. 
\end{Prop}
\begin{proof}[Proof of Proposition \ref{PropK0}] 
To satisfy the inequality $2/d<\eta^2$,
we choose $K_0> 2N/\eta^2$. 
 Applying Theorem \ref{ThmFedja}, we get a full measure set in $(\T^N)^K,\ K>K_0,$ each point of which generates a set $S$ of dimension $d=K/N-\epsilon$ satisfying $2/d<\eta^2$, where $\epsilon$ is arbitrarily small. Next, removing further a zero measure set to guarantee that the vectors rationally generate $\T^N$, we get the full measure set $\mathcal R_{N,K}$ as claimed. 
\end{proof}
The next proposition gives the choice of $\eta$ in Proposition \ref{PropK0}, and will give the $C^{1+}$ regularity of $h$ along the one-dimensional leaves of a foliation after applying Corollary \ref{CorKO} and Proposition \ref{PropMoC}. 
\begin{Prop}\label{PropMain} Suppose
\begin{enumerate}
\item the abelian group $\mathcal A(<\mathcal H^r)$ is generated by $$\{ T_{i, j}\ |\ \ i = 1,\ldots, N,\  j = 1,\ldots, K,\ T_{i,j}T_{i',j'}=T_{i',j'}T_{i,j}\};$$
\item there is an $\eta$-bi-H\"older conjugacy $h$ such that $T_{i, j}(x) = h^{-1}(h(x) +\rho_{i, j})$;
\item there is a $\{T_{i,j}\}$-invariant foliation $\cF$ into one-dimensional $C^1$ leaves $\cF(x)$  with tangential distributions $E(x),\ x\in \T^N$ that is $\eta$-H\"older
in $x$. Denote by $v(x)$ a unit vector field tangent to $\cF(x)$, $x\in \T^N$;
\item the set $S$ generated by the rotation vectors $\rho_{i, j}$ has dimension $d$, with $2/d < \eta^2$.
\end{enumerate}
For any $\gamma\in S$, we denote $$T_\gamma:=\prod_{j=1}^K\prod_{i=1}^N T_{i,j}^{q_{i,j}}
$$
where $q_{i, j}\in \Z$ are the coefficients in the linear combination of $\gamma$, i.e. $\gamma=\sum_{j}\sum_i q_{i,j}\rho_{i,j}$.

Then for all 
$\gamma\in S$ with $\|\gamma\|$ small enough, we have
$$|\|D_xT_\gamma v(x)\|_{C^0}- 1| \leq \mathrm{const.}\|\gamma\|^\nu$$
where $\nu\leq \eta^2 - 2/d $.
\end{Prop}

We defer the proof to Section \ref{SSEllipticRegularity}. 
We next cite the following well-known theorem of Journ\'e.

\begin{Thm}[\cite{J}]\label{ThmJourne}
 Suppose $\cF^1$, $\cF^2$ are two transverse continuous foliations a manifold $M$ with uniformly $C^{n,\nu}$ leaves.  Suppose that a continuous function $u : M \to \R$ is uniformly $C^{n,\nu}$ when restricted to each local leaf $\cF^1_\eps (x), \cF^2_\eps (x),\ x \in M$. Then $u$ is $C^{n,\nu}$ on $M$.
\end{Thm}
In the $2$-dimensional case, we apply Proposition \ref{PropMain} and Proposition \ref{PropMoC} to get that $h$ is $C^{1+}$ along the stable and unstable foliations of the Anosov diffeomorphism $A$. Applying Theorem \ref{ThmJourne}, we get that $h$ is $C^{1+}$ on $\T^2$. 

An application of the next result completes the proof of Theorem \ref{Thm2}. More details of the proof of Theorem \ref{Thm2} will be given in Section \ref{SSThm2}. 

\begin{Thm}[\cite{LMM, Ll}] \label{ThmdelaLlave} Suppose f and g are two $C^r,\ r > 1,$ Anosov diffeomorphims
$\T^2$ that are topologically conjugated by $h$, i.e. $f \circ h = h\circ g$. Suppose the periodic data of $f$
and $g$ coincide, namely, $D_{h(x)}f^q$ is conjugate to $D_xg^q$ at every $q$-periodic point $x$ of $g$ for
all $q\in \Z$. Then $h \in C^{r-\eps}$ for $\eps$ arbitrarily small.\end{Thm}

The proof of Theorem \ref{ThmMain} in the $N>2$ case follows from the same general strategy. However, there is some more work needed to show that the conjugacy $h$ sends the one-dimensional leaves $W_i^{u,s}$ to the straight lines parallel to the eigenvectors of $\bar A$. We will give the proof of the $C^{1+}$ regularity of $h$ in Section \ref{SSC1N}.

In dimension three, we get improved regularity (Corollary \ref{Thm3}) by applying the following result of Gogolev in \cite{G2}. 
\begin{Thm}[Addendum 1.2 of \cite{G2}]\label{ThmGogolev}
Suppose $\bar A\in \mathrm{SL}(3,\Z)$ has simple real spectrum and $A\colon  \T^3\to\T^3$ is $C^r,\ r>3$ that is $C^1$ close to $\bar A$. Suppose also that $\bar A$ and $A$ have the same periodic data, then there exists $h\colon  \T^3\to\T^3$ in $C^{r-3-\eps}$ with $h\circ A=\bar A\circ h.$ Furthermore there exists $\kappa\in \Z$, such that if $r\notin(\kappa,\kappa+3)$, then $h\in C^{r-\eps}$, where $\eps$ is arbitrarily small. 
\end{Thm}

\subsection{Elliptic regularity in the presence of invariant distributions}\label{SSEllipticRegularity}
In this section, we prove Proposition \ref{PropMain}. 

\begin{proof} [Proof of Proposition \ref{PropMain}]
Let $T_\gamma$ and $v(x)$ be as in the statement, and let $\mu$ be any ergodic measure of $T_\gamma$. 
We get from Lemma \ref{LmLyap} that for $\mu$-a.e. $x$
\begin{equation}\label{EqLyap}
\lim_k\frac{1}{k}\log\|D_x(T_\gamma)^k(x)v(x)\|=\int\log\|D_xT_\gamma(x)v(x)\|\,d\mu=0.
\end{equation}
This shows that $\log \|D_xT_\gamma(x)v(x)\|$ vanishes at some point on $\T^N$.

To simplify notation, we reindex the $T_{i, j}$ appearing in $T_\gamma$
 by $T_1,\ldots,T_{\|\gamma\|_w}$, and write
$T_\gamma=\prod_{i(\gamma)=1}^{\|\gamma\|_w}T_{i(\gamma)}$. (Due to the commutativity of the $T_{i, j}$'s, the ordering of the $(i,j)$ appearing in $i(\gamma)$ does not matter). We also write $T_{i(\gamma)}x = x_{i(\gamma)}$.

Consider  the  $\eta$-H\"older function $\ell_{i(\gamma)}(x) :=\log \|D_xT_{i(\gamma)+1}v(x)\|$. Invariance
of the distribution $E$ implies that
\begin{equation*}
\begin{aligned}
\log\|D_xT_\gamma v(x)\|&=\log\left\|D_x \left(\prod_{i(\gamma)=1}^{\|\gamma\|_w}T_{i(\gamma)}\right)v(x)\right\|=\sum_{i(\gamma)=1}^{\|\gamma\|_w}\log\|D_{x_{i(\gamma)}}T_{i(\gamma)+1}v(x_{i(\gamma)})\|\\
&=\sum_{i(\gamma)=1}^{\|\gamma\|_w}\ell_{i(\gamma)}(x_{i(\gamma)}).
\end{aligned}
\end{equation*}

To
prove the lemma, it suffices to restrict attention to a neighborhood of $\gamma=0$.
 We consider a dyadic decomposition of a small neighborhood of $0$ by
$$D_m =\{ \gamma\in S\ |\  c2^{-d(m+1)/2} < \|\gamma\|\leq c2^{-dm/2}\},$$
where $c$ is the constant in Definition \ref{DefDim}. Next, for $D_m$, we introduce a $c2^{-dm}$-net by defining
$$S_m:= \{ \gamma \in D_m\ |\ \|\gamma
\|_w\leq 2^{m}\}.$$

The remaining proof is split into two steps. In the first step, we prove the following

{\bf Claim 1:} {\it For any $\gamma\in S_m$, we have $$|\log\|D_xT_\gamma v(x)\||\leq \mathrm{const.}\|\gamma\|^{\eta^2-2/d},\ \forall \ x\in \T^N.$$}
\begin{proof}[Proof of Claim 1] First, by \eqref{EqLyap}, for any given $\gamma\in S_m$, there exists $y\in \T^N$ such that $\log \|D_yT_\gamma v(y)\|=0$. Next, it follows from the definition of the dimension of the set $S$ that there exists $ \dt\in S$ with 
$$\|\dt\|_w\leq\|\gamma\|_w,\quad \|\dt+\bar y-\bar x\|\leq c\|\gamma\|_w^{-d},\ \bar x=h(x),\ \bar y=h(y).$$
We denote $y_\dt=T_\dt y$ and $y_\gamma=T_\gamma y.$

Since $h$ is bi-H\"older, we have for all $i(\gamma)=0,1,\ldots,\|\gamma\|_w-1$ and $i(\dt)=0,1,\ldots,\|\dt\|_w-1$, the following estimates
$$\|x_{i(\gamma)}-(y_\dt)_{i(\gamma)}\|\leq\mathrm{const.}\|\gamma\|_w^{-d\eta},\; \hbox{and }\,\|y_{i(\dt)}-(y_\gamma)_{i(\dt)}\|\leq\mathrm{const.}\|\gamma\|^{\eta}.$$
Next we estimate $\log\|D_xT_\gamma v(x)\|$ as follows
\begin{equation}\label{EqMain}
\begin{aligned}
&|\log\|D_xT_\gamma v(x)\||\\
&=|\log\|D_xT_\gamma v(x)\|-\log\|D_y(T_\gamma T_\dt) v(y)\|+\log\|D_y(T_\dt T_\gamma) v(y)\||\\
&=|\log\|D_xT_\gamma v(x)\|-\log\|D_{T_\dt y} T_\gamma v(T_\dt y)\|\\
&\ -\log\|D_y T_\dt v(y)\|+\log\|D_{T_\gamma y} T_\dt v(T_\gamma y)\|+\log\|D_yT_\gamma v(y)\||\\
&=|\log\|D_xT_\gamma v(x)\|-\log\|D_{ T_\dt y}T_\gamma v(y_\dt )\|\\
&\ -\log\|D_yT_\dt v(y)\|+\log\|D_{T_\gamma y} T_\dt v(y_\gamma)\||\\
&=\Big|\sum_{i(\gamma)=1}^{\|\gamma\|_w-1}(\ell_{i(\gamma)}(x_{i(\gamma)})-\ell_{i(\gamma)}((y_\dt)_{i(\gamma)}))+\sum_{i(\delta)=1}^{\|\dt\|_w-1}(\ell_{i(\dt)}(y_{i(\dt)})-\ell_{i(\dt)}((y_\gamma)_{i(\dt)}))\Big|\\
&\leq \mathrm{const.}(\|\gamma\|_w\cdot\|\gamma\|_w^{-\eta^2 d}+\|\dt\|_w\|\gamma\|^{\eta^2})\\
&\leq \mathrm{const.}(2^{m(1-\eta^2 d)}+2^m\|\gamma\|^{\eta^2})\\
&\leq \mathrm{const.}(\|\gamma\|^{2(\eta^2-1/d)}+\|\gamma\|^{\eta^2-2/d})\\
&\leq \mathrm{const.}\|\gamma\|^{\eta^2-2/d}.
\end{aligned}
\end{equation}
\end{proof}
In the second step, we prove the following.

{\bf Claim 2:} {\it  Suppose for any $\gamma_m\in S_m$, we have $\|\log\|D_xT_{\gamma_m}v(x)\|\|_{C^0}\leq \mathrm{const.}\|\gamma_m\|^{\nu}$, for some $\nu>0$ and all $x$.   Then for any $\gamma\in S$, we have $\|\log\|D_xT_\gamma v(x)\|\|_{C^0}\leq \mathrm{const.}\|\gamma\|^{\nu}$.
}
\begin{proof}[Proof of Claim 2]
By the definition of $D_m$ and $S_m$ and Definition \ref{DefDim}, we get that each annulus
$D_m$ in the dyadic decomposition is covered by at least $O(2^{dNm/2})$ balls of radius $c2^{-dm}$ centered
at points in $S_m$.

We claim that 
 for any $\gamma\in S$
 with small norm $\|\gamma\|$, there exists a finite number $\kappa(\gamma)$ and $\{
\gamma_{m_k},\ k=1,2,\ldots,\kappa(\gamma)\}$ satisfying $\gamma_{m_k}\in D_{m_k},\ \mathrm{and\ } m_{k+1}\geq 2m_k$ and $\gamma=\sum_{k=1}^{\kappa(\gamma)}\gamma_{m_k}.$ 

The algorithm is as follows. First find $m$ such that $\gamma\in D_m$. Denote this $m$ by $m_1$ and find $\gamma_{m_1}\in S_{m_1}$ that is closest to $\gamma$. The closest distance is bounded by $c 2^{-dm_1}$. Next consider the vector $\gamma-\gamma_{m_1}$ and repeat the above procedure to it in place of $\gamma$. We see that $\gamma-\gamma_{m_1}\in D_{m_2}$ for some $m_2\geq 2m_1$. This procedure terminates after finitely many steps since $\gamma\in S$ is a finite integer linear combination of the rotation vectors $\rho_i,\ i=1,\ldots, m$.

Next, let $x_{m_i}=\prod_{j=i}^{\kappa(\gamma)}T_{\gamma_{m_j}}(x)$ and $x_{m_{\kappa(\gamma)+1}}=x$. Then
\begin{equation*}
\begin{aligned}
|\log\|D_xT_\gamma v(x)\||&=|\log\|D_x\prod_{i}T_{\gamma_{m_i}} v(x)\||\\
&\leq \sum_i|\log\|D_{x_{m_{i+1}}}T_{\gamma_{m_i}} v(x_{m_{i+1}})\||\\
&\leq \mathrm{const.}\sum_{i=1}^{\kappa(\gamma)}\|\gamma_{m_i}\|^{\nu}.
\end{aligned}
\end{equation*}
By the construction of $D_m$ and $S_m$, for all $\gamma\in S$ , we have that $\frac{1}{2}\|\gamma_{m_1}\|\leq\|\gamma\|\leq 2\|\gamma_{m_1}\|$, and $\|\gamma_{m_k}\|$
decays
exponentially with uniform exponential rate. This gives that $ |\log \| D_xT_\gamma v(x)\||\leq \mathrm{const.}\|\gamma\|^\nu$ for
every $\gamma\in S$ close to zero.
\end{proof}
This completes the proof of Proposition \ref{PropMain}. 
\end{proof}


\section{Proof of the Theorems}\label{SProofs} 
In this section, we prove Theorems \ref{Thm2} and \ref{ThmMain}. 
\subsection{Proof of Theorem \ref{Thm2}}\label{SSThm2}
\begin{proof}[Proof of Theorem \ref{Thm2}]

We first explain how to choose $K_0$ and the open set $\mathcal O$ in the statement of Theorem \ref{Thm2}. We choose $\mathcal O$ to be a $C^1$ neighborhood of $\bar A$ in the set of Anosov diffeomorphisms with simple spectrum.  

By Proposition \ref{PropK0}, in order to determine $K_0$, it is enough to determine $\eta$. 
Given $\bar A$ and an Anosov diffeomorphsm $A\colon \T^2\to\T^2$ homotopic to $\bar A$, Theorem \ref{ThmFranks} provides a  a bi-H\"older map $h$ such that $h\circ A=\bar A\circ h$. The H\"older regularity of the conjugacy $h$ depends on both the spectrum of $\bar A$ and the Mather spectrum of $A$ (\cite{KH}
Theorem 19.1.2), and the H\"older regularity of the invariant distribution $E_i(x)$ of the Anosov diffeomorphism $A$ depends on the Mather spectrum of $A$. We choose $\eta$ to be the minimum of these H\"older exponents.  

For  $K>K_0$,  Proposition \ref{PropK0} supplies a full measure set $\mathcal R_{2,K}$ in $\T^{N\times K}$.
Given $i\colon  \{1,\ldots,K\}\to \{1,2\}$, if the rotation vectors of $(\rho_{i(1),1},\ldots, \rho_{i(K),K})$ lie in $\mathcal R_{2,K}$, then the set $S$ generated by the set of all rotation vectors $\{\rho_{i,j},\ i=1,2,\ j=1,\ldots,K\}$ has dimension $d\in (K/2, K).$ For $K>K_0$, we have $2/d< \eta^2$ by Proposition \ref{PropK0}. Moreover $S$ is dense on $\T^N$. 

Consider now an action $\alpha\colon  \G_{\bar B,K}\to \mathrm{Diff}^r(\T^2)$ with $\alpha(g_0)=A\colon  \T^2\to\T^2$ Anosov, 
and $\al(g_{i,k}), \ i=1,2,\ k=1,\ldots,K$ generating an abelian subgroup action $(\Z^2)^K\to \mathrm{Diff}^r(\T^2)$.
As in the hypotheses of the theorem, assume that the subgroup generated by $\al(g_{1,1})$ and $\al(g_{2,1})$ has sublinear oscillation.  Then
applying Theorem \ref{ThmTop2} to the $\Gamma_{\bar A}$ action generated by $\al(g_0)$, $\al(g_{1,1})$ and $\al(g_{2,1})$, we get a bi-H\"older map $h$ linearizing the $\Gamma_{\bar B}$ action.

We show that the conjugacy $h$ given by Theorem \ref{ThmTop2} also linearizes the whole $\G_{\bar B,K}$ action $\alpha.$ Indeed, for any
diffeomorphism $f$ that commutes with $\al(g_{1,1}),\al(g_{2,1})$, we have 
\[h f h^{-1}(x + \rho_{i,1}) =
h f h^{-1}(x) + \rho_{i,1},\]
for $i=1,2$.
Since the rotation vectors $\rho_{i,1},\ i=1,2,$ rationally generate $\T^N$, by taking Fourier expansions,
we get that $h f h^{-1}$ is a rigid rotation by a constant vector that is the rotation vector of $f$. Thus $h$ conjugates the whole action $\al$ to an affine action by rigid translations.

We next apply Proposition \ref{PropMain}, Corollary \ref{CorKO} and Proposition \ref{PropMoC} to get that the conjugacy $h$ is $C^{1+}$ along the stable and unstable leaves of the Anosov diffeomorphism $A$. By Theorem \ref{ThmJourne}, we get that $h$ is $C^{1+}$ on $\T^2$ and finally by Theorem \ref{ThmdelaLlave}, we get that $h$ is $C^{r-\eps}$, for  $\eps$ sufficiently small.
\end{proof}
\subsection{Proof of Theorem \ref{ThmMain}, the $N$ dimensional case}\label{SSC1N}
The main difficulty in generalizing the above argument to the
$N$-dimensional case is that it is in general unknown if the one dimensional distributions $E_i^u$ (or $E^s_i$) that are
invariant under $DA$ are also invariant under $DT_\gamma$. It is only known that the weakest
stable and unstable distributions $E^u_1$ and $E^s_1$ are invariant under $DT_\gamma$ by Proposition \ref{PropGogolev} (4) and Lemma \ref{LmInvariant}.

We cite the following Lemma from \cite{GKS}.

\begin{Lm}[Proposition 2.4 of \cite{GKS}] \label{LmGKS}Let $A$, $\bar A$ and $h$ be as in Proposition \ref{PropGogolev}. Suppose $h$ is $C^{1+}$ along $\cW^u_{\leq i}$ and $h(\cW^u_j
(x)) =\bar \cW^u_j(h(x))$, $1\leq j\leq i$, then
$$h(\cW^u_{i+1}(x)) =\bar \cW^u_{i+1}(h(x)),\ x \in \T^N.$$
\end{Lm}

Using this lemma, we now prove that $h \in C^{1+}$ in the general case $N >2$.
\begin{proof}[Proof of Theorem \ref{ThmMain}] 
The proof follows the strategy of the proof of Theorem \ref{Thm2} with small modifications to deal with the high dimensionality. 

We first choose $K_0$ and the open set $\mathcal O$ of Anosov diffeomorphisms. Since $\bar A$ is assumed to have simple spectrum, it has a $C^1$ small neighborhood in which the Anosov diffeomorphisms have simple Mather spectrum. We choose such a neighborhood and denote it by $\mathcal O$. We will choose $K_0$ to satisfy $2/d<\eta^2$ using Proposition \ref{PropK0}, where $d$ is the dimension of the set $S$ generated by the rotation vectors $\rho_{i,j}$ and $\eta$ is a lower bound on the  H\"older exponent of the conjugacy $h$ and all the distributions $E^{u,s}_i$, for all  Anosov diffeomorphisms in $\mathcal O$. 

Proposition \ref{PropK0} then gives a full measure set $\mathcal R_{N,K}\subset \T^{N\times K}$. 
We obtain a bi-H\"older conjugacy $h$ that linearizes the whole action $\al\colon  \G_{\bar B,K}\to \mathrm{Diff}^r(\T^N)$ by applying Theorem \ref{ThmTopN} and the argument in the proof of Theorem \ref{Thm2}.

It remains to improve the regularity of $h$ to $C^{1+}$. 
To start, Proposition \ref{PropGogolev} (4) implies that  weakest leaves are preserved:
$h(\cW^u_1 (x)) =\bar \cW^u_1(h(x))$, for all $x$. Next, we apply Lemma \ref{LmInvariant} to get that the
weakest distribution $E^u_1$ is invariant under the abelian group action generated by $\al(g_{i,k}),\ i=1,\ldots,N,\ k=1,\ldots,K$.  Applying Proposition
\ref{PropMain}, Corollary \ref{CorKO}  and Proposition \ref{PropMoC}, we conclude that $h$ is $C^{1+}$ along the weakest leaves
$\cW^u_1(x)$.
Thus the assumption of the Lemma \ref{LmGKS} is satisfied with $i=1$, and we conclude that the
second weakest leaves are preserved $h(\cW^u_2 (x)) =\bar \cW^u_2(h(x))$. We next apply Lemma \ref{LmInvariant}, Proposition
\ref{PropMain}, Corollary \ref{CorKO} and Proposition \ref{PropMoC} to conclude that $h$ is $C^{1+}$ along $\cW^u_{2}$. By Journ\'e's theorem \ref{ThmJourne},
we get that $h$ is $C^{1+}$ along the leaves $\cW^u_{\leq 2}$.

Applying Lemma \ref{LmGKS} inductively in $i$, we conclude $h$ is $C^{1+}$ along the unstable foliation
$\cW^u$. Similarly, we prove that $h$ is  $C^{1+}$ along  $\cW^s$. Then by Journ\'e's theorem \ref{ThmJourne}, we have that $h \in C^{1+}$.\end{proof}

\subsection{Alternative assumptions}\label{SSAlternative}
In this section, we discuss possible alternative assumptions for Theorem \ref{ThmMain}. Our technique developed in Section \ref{SEllipHyp} relies on the existence of foliations by one dimensional leaves that are invariant under the abelian group action. In our proofs, the foliations are provided by the Anosov diffeomorphism. The foliations being invariant under the abelian group action follows from the  existence of a common conjugacy $h$. In other words, we need that the leaves (straight lines)  of the invariant foliation of the toral automorphism $\bar A$ are mapped to the leaves of the invariant foliations of $A$ by the conjugacy $h^{-1}$ (Proposition \ref{PropGogolev}). 
This is true when $N=2$ or in higher dimensions when we assume that $A$ is $C^1$ close to $\bar A$. There are also circumstances under which  Proposition \ref{PropGogolev} can be proved without the $C^1$ smallness assuption. We mention here  two main cases. 

 In \cite{G1}, the author considers an Anosov diffeomorphism $A$ homotopic to a linear map $\bar A$ with simple Mather spectrum and the property that in each connected component of the Mather spectrum, there lies exactly one eigenvalue of $\bar A$. Moreover it is assumed that the invariant distributions $E^{u,s}_i$  form  angles less than $\pi/2$ with the corresponding affine distributions $\bar E^{u,s}_i$ for the linear map $\bar A$. (This assumption guarantees a certain quasi-isometric property of $\cW^s$ and $\cW^u$).  Under these assumptions, the conclusions of  Proposition \ref{PropGogolev} hold \cite{G1}.

 In \cite{FPS}, a similar result is shown assuming that $A$ is isotopic to $\bar A$ along a path of Anosov diffeomorphisms with simple Mather spectrum.

\appendix

\section{Proof of Theorem \ref{ThmFedja}}\label{Appendix} The proof was communicated to us by the user Fedja on MathOverflow \url{http://mathoverflow.net/questions/227817/a-quantitative-kronecker-theorem}.

 We need only  consider matrices $M=(m_{ij}),\ m_{i,j}\in\T,\ i=1,\ldots,N$ and $j=1,\ldots,K$.  So $\mathcal M_{N\times K}(\T)$ is identified with $\T^{N\times K}$ endowed with Lebesgue measure.
  
  Fix a smooth function  $\psi\in C^\infty(\R)$ with supp$\psi \subset (-1, 1),\ \psi\geq 0$ and $\int\psi  = 1$. Let $\eps$ be fixed. We next introduce $r_n = n^\eps n^{- K/N}, \ n\in \N$, and put $\Psi_n(x) = r_n^{-N} \prod_{i=1}^N\psi(x_i/r_n)$ and consider the periodic function $\Phi_{n,y}(x) =\sum_{q\in\Z^N}\Psi_n(x-q-y)$ for each $y\in \T^N$. Then we claim that

{\it Given $\eps>0$, there exists a $\dt > 0$ such that for each $n\in \N$, there exists a set $\mathcal U_n \subset \mathcal M_{N\times K}(\T)$ with
$Leb(\mathcal U_n) < n^{-\dt}$, and for each $M \notin \mathcal U_n$ and any $y\in \T^N$, there exists $p \in \Z^K$ with $\|p\|\leq n$ and $\Phi_{n,y}(Mp)\neq 0$.}

Assuming the claim, considering $n = 2^\ell, \ell\in \N$ and using Borel-Cantelli, we get
$Leb(\limsup_n\mathcal U_n) = 0$. This means that the probability for $M$ lying in infinitely many $\mathcal U_n$ is
zero. This completes the proof of the theorem.

It remains to prove the claim. Decompose $\Phi_{n,y}(x)$ into Fourier series $\Phi_{n,y}(x) =\sum_{k\in \Z^N}c_k(n,y) e^{2\pi i\langle k,x\rangle}.$ Notice that for each $x$, there is only one $q\in \Z^N$ such that $x-q-y\in (-1,1)^N$. It follows that $\|c_k(n,y)\| \leq 1$ for all $k \in \Z^N$ and $c_0(n,y) = 1$ is independent of $n,y$. Moreover, for each $\ell>0$, there exists $C_\ell$ (depending only on $\psi$) such that $\|c_k(n,y)\| \leq C_\ell r_n^{-N-\ell}/\|k\|^\ell$ due to the $C^\infty$ smoothness of $\psi$. 
Next for any matrix $M\in \mathcal{M}_{N\times K}(\T)$ write
$$S_n(M,k):=\sum_{\|p\|_{\infty}\leq n}e^{2\pi i\langle k,Mp\rangle},\quad \Lambda_{n,y}(M):=\sum_{\|p\|_\infty\leq n}\Phi_{n,y}(Mp)=\sum_{k\in\Z^N} c_k(n,y) S_n(M,k).$$
We get $S_n(M,0)=n^K$ and for $\beta>0$ to be determined later
$$\left|\sum_{\|k\|\geq r_n^{-\beta}} c_k(n,y)S_n(M,k)\right|\leq n^K C_\ell r_n^{\beta \ell-N-\ell}.$$
It remains to investigate the sum 
$$\Gamma_{n,r_n}(M):=\sum_{0<\|k\|_\infty\leq r_n^{-\beta}}|S_n(M,k)|=\sum_{0<\|k\|_\infty\leq r_n^{-\beta}}\left|\sum_{\|p\|_\infty\leq n} e^{2\pi i\langle z,p\rangle}\right|,\quad z=M^tk.$$
We use the fact that $|\sum_{\|p\|_\infty\leq n} e^{2\pi i\langle z,p\rangle}|\leq C\prod_{j=1}^K(\min\{n,\|z_j\|^{-1}\})$, where $\|z_j\|$ is the distance from $z_j$ to the nearest integer. Consider a map $F_k\colon  \mathcal{M}_{N\times K}(\T)\to \T^K$ via $F_k(M)=M^tk$, mod $\Z^K$, then $F_k$ pushes forward the Lebesgue measure on $\mathcal M_{N\times K}(\T)$ to a Lebesgue measure on $\T^K$. We immediately get that
$$\int_{\mathcal M_{N\times K}(\T)}|S_n(M,k)|\,d\mathrm{Leb}\leq \int_{\T^K}C \prod_{j=1}^K(\min\{n,\|z_j\|^{-1}\})\,dz\leq C\log^Kn,$$
so there exists a set $\mathcal{U}_n\subset \mathcal{M}_{N\times K}(\T)$ with Leb$(\mathcal{U}_n)\leq n^{-\dt}$ such that we have 
$$\Gamma_{n,r_n}(M)\leq C n^\dt r_n^{-\beta N}\log^K n,\quad \forall\ M\in \mathcal{M}_{N\times K}(\T)\setminus\mathcal{U}_n.$$
Note that this set $\mathcal U_n$ is independent of $y$ since $\Gamma_{n,r_n}(M)$ is. Now we get
$$|\Lambda_{n,y}(M)|\geq n^K-n^K C_\ell r_n^{\beta\ell-N-\ell}-Cn^\dt r_n^{-\beta N}\log^Kn,\quad \forall\ M\in \mathcal{M}_{N\times K}(\T)\setminus\mathcal{U}_n,\ \forall y\in \T^N.$$
We choose $r_n=n^\eps n^{-K/N}$, and $\beta(>1)$ and $\dt(>0)$ sufficiently close to $1$ and $0$ respectively to satisfy the inequality $(\beta-1)K+2\dt<\beta N\eps$ for given $\eps$, and choose $\ell$ large enough to satisfy $(\beta-1)\ell>N$.  
Hence $|\Lambda_{n,y}(M)|\geq \frac{1}{2}n^K$. This completes the proof of the claim hence the theorem. \qed

\section{Affine action and the simultaneous Diophantine condition}\label{AppDiop}
In this section, we discuss the assumption in Theorem \ref{ThmLocal} on $\bar A$ and $\brho$. We show here how to ensure that assumption \eqref{EqCommute} and the Diophantine assumption $\brho$ are satisfied simultanously. 

Given $\bar A$, we solve equation \eqref{EqCommute} for $\brho$. Lifting \eqref{EqCommute} to $\R^N$, we get the following equation 
\begin{equation}\label{EqCommuteInhom}\bar A\brho=\brho\bar B+\mathbf P, \quad \mathbf P\in \Z^{N\times N}.\end{equation}

As usual, we first set $\mathbf P=0$ and consider the homogeneous equation. 

The following facts can be found in \cite{HJ}, Theorem 4.4.14. 
\begin{Prop}Suppose $\bar A,$ $\bar B\in \mathrm{SL}(n,\Z)$. 
\begin{enumerate}
\item If the sets of spectrum of $\bar A$ and $\bar B$ do not intersect, then the homogeneous equation has zero solution and the inhomogeneous equation \eqref{EqCommuteInhom} 
is solvable with only rational solutions. In this case, the affine action can never be faithful. 

\item If the sets of spectrum of $\bar A$ and $\bar B$ do intersect and either $\bar A$ or $\bar B$ is diagonalizable over $\C$. Denote the common eigenvalues by $\lambda_1,\ldots,\lambda_k$, the eigenvector for $\bar A$ associated to $\lambda_i$ by $a_{i,1},\ldots, a_{i,n_i}$ and the eigenvector for $\bar B^t$ associated to $\lambda_i$ by $b_{i,1},\ldots, b_{i,m_i}$.  Then the null space of $\brho\mapsto \bar A\brho-\brho\bar B$ is the span of 
$$\{a_{i, j}\otimes b_{i,\ell},\quad j=1,\ldots,n_i,\ \ell=1,\ldots,m_i,\quad i=1,\ldots,k\}.$$
\end{enumerate}
\end{Prop}
\begin{proof}
The first item follows from Theorem 4.4.6 of \cite{HJ} using the properties of Kronecker product.  

For the second statement, by Theorem 4.4.14 of \cite{HJ}, the null space of the map $\brho\mapsto \bar A\brho-\brho\bar B $ has dimension $\sum_i n_i\times m_i$. It is clear that each matrix $a_{i, j}\otimes b_{i,\ell} $ where $j=1,\ldots,n_i$ and $\ell=1,\ldots,m_i$ lies in the kernal of $\brho\mapsto \bar A\brho-\brho\bar B$, and these matrices are linearly independent, so we get the second statement. 
\end{proof}
In the 2D case, suppose $\mathrm{tr}\bar A=\mathrm{tr}\bar B$ and $|\mathrm{tr}\bar A|>2$, then $\bar A$ and $\bar B$ share the same spectrum $\lambda$ and $1/\lambda$ for some $|\lambda|>1$. The zero space of $\brho\mapsto \bar A\brho-\brho\bar B$ is then spanned by $u_{\bar A}\otimes u_{\bar B^t}$ and 
$u_{\bar A^{-1}}\otimes u_{(\bar B^t)^{-1}}$, where $u_{\bar A}$ is the eigenvector corresponding to the eigenvalue $\lambda$. Similarly for others. If $\mathrm{tr}\bar A\neq \mathrm{tr}\bar B$, we get that the zero space of $\brho\mapsto \bar A\brho-\brho\bar B$ is zero. 

If some of the $a_{i,j}$ is Diophantine, then the simultaneous Diophantine condition is satisfied automatically. We next focus on the special case of $\bar A=\bar B$, where the simultaneous Diophantine condition is more explicit. 
We recall a fact and definition from linear algebra:
\begin{Lm}[Corollary 4.4.15 of \cite{HJ}] \label{LmHJ} Let $A\in M_N(\R)$ where $M_N(\R)$ is the set of $N\times N$ matrices with entries in $\R$.
The set of matrices in $M_N(\R)$ that commute with $A$ is a subspace of $M_N(\R)$ with dimension at least $N$. The dimension is equal to $N$ if and only if $M$ is non-derogatory, i.e. each eigenvalue of $A$ has geometric multiplicity exactly 1.   Thus if $A$ is nonderogary, the centralizer $Z(A)$ of $A$ is
\[Z(\bar A) = \mathrm{span}_\R\{\id, A,\ldots, A^{N-1}\}.\]
\end{Lm}
If  $\bar A$ is non-derogatory, then for any $\brho$ satisfying $\bar A\brho=\brho\bar A $, we can thus write each $\brho\in Z(\bar A)$ as a linear combination
$\brho=\sum_{i=1}^N a_iA^{i-1}$, where $a = (a_1,\ldots, a_N)\in \R^N$.

\begin{Lm}\label{LmSDC} 
Let $\boldsymbol \rho=\sum a_i \bar A^{i-1}$ for some $a=(a_1,\ldots,a_N)$ and $\bar A^{i-1}\in \mathrm{SL}(N,\Z)$. Suppose the nonvanishing entries of $a$ form a vector $a'\in \R^k$,  $1\leq k\leq N$ satisfying the Diophantine condition: there exist $C,\tau>0$ such that 
$$|\langle a',m\rangle|\geq \frac{C}{|m|^\tau},\quad \forall\ m\in \Z^k\setminus\{0\}.$$
Then the columns of $\brho$, denoted by $\rho_1,\ldots, \rho_N,$ satisfy the simultaneous Diophantine condition for some $C' > 0$, i.e. 
\begin{equation}\label{EqSDC}
 \max_{1\leq j\leq N}\{|\langle n,\rho_j\rangle|\}\geq\frac{C'}{\|n\|^{\tau}},\quad \forall\ n\in \Z^N\setminus\{0\}.
 \end{equation}
\end{Lm}
\begin{proof}

Denote by $v^i_j$ is the $j$-th column of $\bar A^{i-1},\ i,j = 1,\ldots, N$, and by $\rho_j$ the $j$-th column of $\boldsymbol\rho$. Hence we have $\rho_j =\sum_i a_iv^i_j$. 

Assume the vector formed by the non vanishing entries of $a$ satisfies the Diophantine condition and
denote by $\mathcal I$ the set of indices of the non vanishing entries of the vector $a$, then we have for each $j$
\begin{equation}
\begin{aligned}
|\langle n,\rho_j\rangle|=\left|\sum_{i=1}^N a_i\langle n,v_j^i\rangle\right|&\geq \frac{C}{(\sum_{i\in \mathcal{I}}|\langle n,v_j^i\rangle|)^{\tau}}\\
&=\frac{C}{\|n\|^{\tau} (\sum_{i\in\mathcal{I}} \left|\langle \frac{n}{\|n\|},v_j^i\rangle\right|)^{\tau}}\\
&\geq \frac{C}{\|n\|^{\tau}(\sum_{i\in\mathcal{I}}\|v_j^i\|)^{\tau}}
\end{aligned}
\end{equation}
if $\langle n,v_j^i\rangle\neq 0$ for some $i\in\mathcal{I}$. 

 To show that the simultaneous Diophantine condition holds for
$\rho_1,\ldots,\rho_N$, it remains to show that for each $u \in \S^{N-1}$, there exist $i\in\mathcal{I},$ $j\in \{1, 2, \ldots, N\}$,
such that $\langle u, v^i_j\rangle\neq 0$. This follows from the non-degeneracy of $\bar A$. We fix any $i \in \mathcal{I}$, then
$v^i_j,\ j = 1, 2, \ldots, N$, form the matrix $\bar A^{i-1}$ which is non-degenerate. Hence the vectors $v^i_j,\ j =
1, 2,\ldots, N$, are linearly independent. The compactness of $\S^{N-1}$ implies that there does not
exists $u\in \S^{N-1}$ that is simultaneously orthogonal to all of $v^i_j,\ j = 1, 2,\ldots, N$.
\end{proof}

Next, in order to solve the inhomogeneous equation $\bar A\brho=\brho\bar A +\mathbf P$, it is enough produce a particular solution for given $\mathbf P\in \Z^{N\times N}$ in addition to the general solutions to the homogeneous equation. Note that the \eqref{EqCommuteInhom} might not be solvable for some $\mathbf P$. We have the following result.

\begin{Thm}[Theorem 4.2.22 of \cite{HJ}] Given matrices $A,B,C\in M_N(\R)$. Then there exists some $X\in M_N(\R)$ solving the equation $A X-XB=C$ if and only if the matrices $\left[\begin{array}{cc}
A&C\\
0&B
\end{array}\right]$ and $\left[\begin{array}{cc}
A&0\\
0&B
\end{array}\right]$are similar. 
\end{Thm}
So to solve \eqref{EqCommuteInhom}, the necessary and sufficient condition is the similarity of the matrices $\left[\begin{array}{cc}
\bar A&\mathbf P\\
0&\bar A
\end{array}\right]$ and $\left[\begin{array}{cc}
\bar A&0\\
0&\bar A
\end{array}\right]$.
Given a particular solution $\brho_*(\mathbf P)$ of \eqref{EqCommuteInhom}.  If $\bar A$ is nonderogatory, then the general solution of \eqref{EqCommuteInhom} can be written as $$\brho=\sum a_i \bar A^{i-1}+\brho_*$$ for some $a=(a_1,\ldots,a_N)$, if $\brho_*$ happens to be rational, then $\brho$ is simultaneously Diophantine, if the nonvanishing entries of $a$ form a Diophantine vector. 

\section{Affine actions and vanishing Lyapunov exponents}\label{SAffine}

In this appendix, we prove the results in Section \ref{SSAffine}. 
Proposition \ref{p=action} is verified straightforwardly from the group relation. We prove Proposition \ref{PropFaithful} and Proposition \ref{PropUniqueAffine}.

\begin{proof}[Proof of Proposition \ref{PropFaithful}]
The proof of $\Longrightarrow$ is easy. We only prove $\Longleftarrow$ here. Suppose the action is not faithful. Then there exist $\gamma_1,\gamma_2\in \Gamma_{\bar B}$ with $\gamma_1\neq \gamma_2$ but $\al(\gamma_1)=\al(\gamma_2)$. Using the group relation, we first rewrite $\gamma_i$ in the form $\gamma_i = g_0^{m_i}g^{p_i},\ i=1,2$,  where $g^p=g_1^{p_1}\ldots g_N^{p_N}$.  We can deduce an equation of the form $\al(g_0)^m=\al(g^p)$ with $m=m_1-m_2$ and $p=p_1-p_2$ from $\al(\gamma_1)=\al(\gamma_2)$. We pick any rational point $x$ on $\T^N$ and note that $\al(g_0)^mx$ is rational but $\al(g^p)x$ is irrational unless $p=0$ by the linaer independence of $\brho$. If $p=0$, then $m=0$ since $\bar A$ is not of finite order. This implies that $\gamma_1=\gamma_2.$
\end{proof}

\begin{proof}[Proof of Proposition \ref{PropUniqueAffine}]
Suppose we have two affine actions $\bar\al=\bar\al(\bar A,\brho)$ and $\bar\al'=\bar\al(\bar A,\brho')$ conjugate by a homeomorphism $h$ of the form $h(x)=x+\tilde h(x),\ x\in \T^N$ where $\tilde h$ is $\Z^N$-periodic. We want to show that $\brho=\brho'$. 
Denote by $\rho_j$ and $\rho_j'$ the $j$-th column of $\brho$ and $\brho'$ respectively. We have $$h(x+\rho_j)=h(x)+\rho_j',\quad j=1,\ldots,N.$$
This is equivalent to 
$$\rho_j+\tilde h(x+\rho_j)=\tilde h(x)+\rho_j',\quad j=1,\ldots,N.$$
Integrating over $\T^N$, we get that $\int_{\T^N}\tilde h(x+\rho_j)\,dx=\int_{\T^N}\tilde h(x)\,dx$, hence $\rho_j=\rho_j'$. 
\end{proof}
\begin{Prop}
Suppose $\bar B\in \mathrm{SL}(N,\Z)$ has no eigenvalue 1. Then for any action of $\al: \Gamma_{\bar B}\to \mathrm{Diff}^r(\T^N),\ r>1$, all the Lyapunov exponents of $\al(g_i),\ i=1,2\ldots,N,$ are zero with respect to any invariant measure. 
\end{Prop}
\begin{proof}
We use the following Zimmer amenable reduction theorem.
Fix a group action $\al\colon  \Gamma\to \mathrm{Diff}(M)$, and let  $\phi \colon(\al,M)\to \mathrm{GL}(N,\R)$ be a cocycle, meaning that
\[\phi(\al(\gamma_1)\al(\gamma_2),x)=\phi(\al(\gamma_1),\al(\gamma_2)x)\phi(\al(\gamma_2),x),\]
for all $x\in M$ and $\gamma_i\in \Gamma$. 
We say that $\phi$ is cohomologous to  another cocycle $\psi$ if there exists a measurable map $h\colon  M\to \mathrm{GL}(N,\R)$ such that $$\phi(\al(\gamma),x)h(x)=h(\al(\gamma)x)\psi(\al(\gamma),x),\quad\forall x\in M,\ \forall\ \gamma\in \Gamma.$$

\begin{Thm}[Theorem 1.8 of \cite{HuK}]\label{ThmZimmer}
Let $\al\colon  \Gamma\to \mathrm{Diff}(M)$ be an amenable group action and $\phi: (\al,X)\to \mathrm{GL}(N,\R)$ a cocycle.  Then there exists a cocycle $\psi\colon (\al,M)\to \mathrm{GL}(N,\R)$ that is cohomologous $\phi$ and such that there exists a partition of $X=\cup_{i=1}^{2^N}X_i$ and $\psi\colon  (\al,X_i)\to H_i$, where $H_i$ is one of the $2^n$ conjugacy classes of maximal amenable subgroups of $\mathrm{GL}(N,\R)$ of the form $\left[\begin{array}{cccc}
A_1&*&\ldots&*\\
0&A_2&*&*\\
0&0&\cdot &\cdot\\
\cdot&\cdot&\cdot&\cdot\\
0&\ldots&0& A_k
\end{array}\right]$, each $A_i$ is $n_i\times n_i$ with $\sum_{i=1}^kn_i=N$ and is of the form of a scalar times an orthogonal matrix. 
\end{Thm} 
We next cite the following result on the Weyl chamber of the $\Z^N$ actions on a compact manifold. 
\begin{Thm}[Proposition 2.1 of \cite{FKS}]\label{ThmWeyl}
Suppose $\mu$ is an ergodic measure for the action $\beta\colon  \Z^N\to \mathrm{Diff}^r(M),\ r> 1$. Then there are finitely many linear functionals $\chi: \ \Z^N\to \R$, a set $\mathcal P$ of full measure and a $\beta$-invariant measurable splitting of the tangent bundle $T_xM=\oplus E_\chi(x)$, $x\in \mathcal P$ such that for all $a\in \Z^N$ and $v\in E_\chi$, the Lyapunov exponent of $v$ is 
$$\lim_{n\to\pm\infty} n^{-1}\log\|D\beta(a^n)(v) \|=\chi(a).$$
\end{Thm}

With the two results, we give the proof of the proposition. Without loss of generality, we assume $\mu$ is an ergodic measure for the action. A general invariant measure can be decomposed into averages of ergodic measures. From the group relation we obtain
$$ A T^p=T^{\bar B^tp} A.$$
where we have $\alpha(\bar B)=A,\ \al(g_i)=T_i$ and $T^p=\prod_{i=1}^N T_i^{p_i}, \ p=(p_1,\ldots, p_N)\in \Z^N$. 

Applying Theorem \ref{ThmZimmer} to the cocycle $D\al$, we get a measurable map $h\colon  \T^N\to \mathrm{GL}(N,\R)$ such that $D_x\hat\al(g)(x):=h(\al(g)x) D_x\al(g)h^{-1}$ is of the form $H_i$ as in Theorem \ref{ThmZimmer}.  Thus \begin{equation}\label{EqZimmer}
D\hat AD\hat T^p=D\hat T^{\bar B^tp} D\hat A,\end{equation}

Since $h$ is only known to be measurable, we denote by $Y$ the zero measure set of points where $h$ is unbounded, and by $X:=\T^N\setminus (\cup_{g\in \Gamma}\al(g)^{-1}(Y))$, which has full measure. For each $k\in \N$, we introduce the set $X_k:=\{x\in X\ |\ \|h(x)\|\leq k,\ \|h(x)^{-1}\|\leq k\}.$ By Poincar\'e recurrence, for $\mu$-a.e. $x\in X_k$, the $T^p$- and $T^{\bar B^tp}$-orbits of $x$ will return to $X_k$ infinitely often. We pick  such an $x_*\in X_k$ and we get that $A(x_*)\in X_{k'}$ for some $k'$. Next we apply Poincar\'e recurrence to both $T^p$ and $T^{\bar B^tp}$ to obtain a subsequence $\{n_i\}\subset\N$ such that $$T^{n_ip}x_*\in X_k,\quad T^{\bar B^t n_ip}(Ax_*)\in X_{k'}.$$
This implies $AT^{n_ip}x_*=T^{\bar B^t n_ip}(Ax_*)\in X_{k'}$. This gives the estimates 
$$\|D\hat A(T^{n_ip}x_*)\|=\| h(AT^{n_ip}x_*) DA(T^{n_ip}x_*)h^{-1}(T^{n_ip}x_*)\|\leq kk'\|DA\|_{C^0}.$$
Similarly, we estimate
$$\|D\hat T^{n_ip}(x_*)\|\leq k^2 \|DT^{n_ip}\|_{C^0},\quad \|D\hat T^{\bar B^t n_ip}(Ax_*)\|\leq k'^2 \|DT^{\bar B^t n_ip}\|_{C^0}.$$
By Theorem \ref{ThmZimmer}, since each $\hat\al(g),\ g\in \Gamma_{\bar B}$ has the form of $H_i$, we consider only the diagonal blocks. Suppose $D\hat A(x)$ has diagonal blocks $\mathsf a_1(x),\ldots,\mathsf a_j(x)$, and $D\hat T_i(x)$ has diagonal blocks $\mathsf t_{i,1}(x),\ldots,\mathsf t_{i,j}(x)$, where $ \mathsf a_\ell$ and $\mathsf t_{i,\ell}$ are $n_\ell\times n_\ell$ and $\sum_{\ell=1}^j n_\ell=N$. Similarly, we denote the diagonal blocks of $D\hat T^p$ by $\{\mathsf t_\ell^p\}$. 
We further denote $\lambda(\mathsf a_\ell)$ and $\lambda(\mathsf t^p_{\ell})$ the modulus of the scalar part of $\mathsf a_\ell$ and $\mathsf t^p_{\ell}$ respectively. 

Equation \eqref{EqZimmer} gives the following on the diagonal
\begin{equation}\label{EqDiag}\mathsf a_\ell\mathsf t_{\ell}^{n_ip}=\mathsf t_{\ell}^{\bar B^tn_ip}\mathsf a_\ell,\quad\mathrm{and}\quad  \lambda(\mathsf a_\ell)\lambda(\mathsf t_{\ell}^{n_ip})=\lambda(\mathsf t_{\ell}^{\bar B^tn_ip}) \lambda(\mathsf a_\ell).\end{equation}
We take log and divide by $n_i$ and let $n_i\to \infty$. Since $\lambda(a_j)$ is bounded by $kk'\|DA\|$, we have that $\lim\frac1n\log(a_j)\to 0.$  Let $\mu_{i,j}:=\lim_{n_i}\frac{\log\lambda(\mathsf t_{i,j}^{n_ip})}{n_i}$, whose existence is given by the ergodic theorem, and denote by  $M$ the matrix $(\mu_{i,j})$. We will show below that each row of $M$ gives rise to a Lyapunov functional $\chi_j$,and hence by Theorem \ref{ThmWeyl} and equation \eqref{EqDiag} we have $\chi_j(p)=\chi_j( \bar B^tp)$. Choosing $p$ to be of the form $n(1,0,\ldots,0),\ n(0,1,0,\ldots,0),\ldots,n(0,\ldots,0,1)$, we get the following$$M=M\bar B^t,\quad {\rm i.e.}\quad M(\bar B^t-\mathrm{Id})=0.$$
Since $\bar B$ does not have eigenvalue $1$, the only solution is $M=0$ so all the Lyapunov exponents $\mu_{i,j}$ are $0$.   

It remains to show that each row of $M$ is a Lyapunov functional. We apply Theorem \ref{ThmWeyl} to the abelian group $\{T^p,\ p\in \Z^N\}$. For the linear functional $\chi$ and invariant splitting $\oplus E_\chi(x)$ of $\{T^p\}$, we get that the splitting $\oplus h(x)E_\chi(x)$ is invariant under $\{D\hat T^p\}$. So for each $v\in E_\chi(x)$, the Lyapunov exponent of $T^p$ at point $x\in X$ along the vector $v$ is given by $\lim\frac1n\log\|D_xT^{np}v\|=\chi(p)$ and for $h(x)v\in h(x)E_\chi(x)$, the Lyapunov exponent of $\hat T^p$ at the point $x$ along the vector $h(x)v$ is also $\chi(p)$. This shows that $DT^p$ and $D\hat T^p$ share the same Lyapunov functional. It remains to identify the Lyapunov exponents of each $D\hat T_j$ as $\{\mu_{i,j}\}$. Since $D\hat T^p$ has the form of $H_i$ in Theorem \ref{ThmZimmer}, we get that the invariant splitting can be constructed explicitly and inductively. We denote by $e_1,\ldots, e_n$ the standard basis vectors of $\R^N$. We first denote $\mathcal E_1=\mathrm{span}\{e_1,\ldots,e_{n_1}\}$. From the normal form in Theorem \ref{ThmZimmer}, it is clear that $\mu_{i,1}$ is the Lyapunov exponent $\lim\frac1n\log\|D\hat T_i(x_*)v\|$ for any $v\in \mathcal E_1$. Therefore $\mathcal E_1$ is one summand in the splitting $\oplus h E_\chi$ and $h^{-1}\mathcal E_1$ is one summand in the splitting $\oplus  E_\chi$. The second Lyapunov exponent $\mu_{i,2}$ is found by acting $D\hat T_i$ on the quotient $\R^N/ \mathcal E_1=(\oplus h E_\chi)/ \mathcal E_1$, equivalently by acting $DT_i$ on the quotient $\R^N/ h^{-1}\mathcal E_1=(\oplus E_\chi)/ h^{-1}\mathcal E_1$. We denote by $\mathcal E_2=\mathrm{span}\{e_{n_1+1},\ldots,e_{n_1+n_2}\}$. From the normal form in Theorem \ref{ThmZimmer}, we see that $\mathcal E_2/\mathcal E_1$ is the invariant subspace for the action of $D\hat T_i$ on the quotient $\R^N/ \mathcal E_1$. This implies that $\mathcal E_2/\mathcal E_1$ is one summand in the quotient splitting $\R^N/ \mathcal E_1=(\oplus h E_\chi)/\mathcal E_1$ and equivalently $h^{-1}\mathcal E_2/h^{-1}\mathcal E_1$ is invariant under the action of $DT$ in the quotient space $\R^N/ h^{-1}\mathcal E_1=(\oplus  E_\chi)/h^{-1}\mathcal E_1$, therefore is a summand in the quotient space. This shows that $\mu_{i,2}$ as the Lyapunov exponent of the quotient action $D\hat T_i$ on the quotient space $\mathcal E_2/\mathcal E_1$ is also the Lyapunov exponent of the quotient action of $DT_i$ on the quotient space $h^{-1}\mathcal E_2/h^{-1}\mathcal E_1$, therefore is one Lyapunov exponent of $DT_i$. 

 Inductively, we find all the Lyapunov exponents $\{\mu_{i,j}\}$. For each $j$, the vector $(\mu_{1,j},\ldots,\mu_{N,j})$ gives rise to a Lyapunov functional $\chi_j$. 


\end{proof}
\section*{Acknowledgment}
A.W. is supported by NSF grant  DMS-1316534.
J. X. is supported by the significant project 11790273 National Natural Science Foundation of China and  Beijing Natural Science Foundation (Z180003). We would like to thank Sebastian Hurtado, Kostya Khanin and Pengfei Zhang for  helpful discussions.

\end{document}